\documentclass[11pt]{amsart}
\usepackage{amssymb,amsfonts,amsmath,amscd}
\usepackage{verbatim}
\usepackage[all]{xy}
\usepackage{hyperref}

\newtheorem{theorem}{Theorem}[section] 
\newtheorem{lemma}[theorem]{Lemma}     
\newtheorem{corollary}[theorem]{Corollary}
\newtheorem{proposition}[theorem]{Proposition}

\newtheorem{definition}[theorem]{Definition}

\newtheorem{remark}[theorem]{Remark}

\newtheorem{example}[theorem]{Example}
\newtheorem{notation}[theorem]{Notation} 

\newcommand{\fa}{\mbox{$\mathfrak{a}$}}
\newcommand{\fb}{\mbox{$\mathfrak{b}$}}
\newcommand{\fc}{\mbox{$\mathfrak{Q}$}}
\newcommand{\fm}{\mbox{$\mathfrak{m}$}}
\newcommand{\fp}{\mbox{$\mathfrak{p}$}}

\newcommand{\bn}{\mbox{$\mathbb{N}$}}
\newcommand{\bz}{\mbox{$\mathbb{Z}$}}
\newcommand{\bq}{\mbox{$\mathbb{Q}$}}
\newcommand{\br}{\mbox{$\mathbb{R}$}}
\newcommand{\bc}{\mbox{$\mathbb{C}$}}
\newcommand{\ba}{\mbox{$\mathbb{A}$}}
\newcommand{\bm}{\mbox{$\mathcal{M}$}}
\newcommand{\bl}{\mbox{$\mathcal{L}$}}
\newcommand{\bt}{\mbox{$\mathcal{T}$}}

\newcommand{\rad}{\mbox{${\rm rad}$}}

\newcommand{\height}{\mbox{${\rm height}$}}
\newcommand{\grade}{\mbox{${\rm grade}$}}

\newcommand{\rank}{\mbox{${\rm rank}$}}
\newcommand{\ch}{\mbox{${\rm char}$}}
\newcommand{\diag}{\mbox{${\rm diag}$}}
\newcommand{\trdeg}{\mbox{${\rm trdeg}$}}
\newcommand{\nulls}{\mbox{${\rm Nullspace}$}}
\newcommand{\adj}{\mbox{${\rm adj}$}}
\newcommand{\hull}{\mbox{${\rm Hull}$}}
\newcommand{\fit}{\mbox{${\rm Fitt}$}}

\title{The primary components of positive critical binomial ideals}

\author{Liam O'Carroll and Francesc Planas-Vilanova}

\date{\today}


\keywords{Herzog-Northcott ideal; positive critical binomial ideal;
  Smith Normal Form.\\ The authors gratefully acknowledge financial
support from the research grant MTM2010-20279-C02-01 and the
Universitat Polit\`ecnica de Catalunya during the development of this
research.}

\begin{document}

\begin{abstract}
A natural candidate for a generating set of the (necessarily prime)
defining ideal of an $n$-dimensional monomial curve, when the ideal is
an almost complete intersection, is a full set of $n$ critical
binomials. In a somewhat modified and more tractable context, we prove
that, when the exponents are all positive, critical binomial ideals in
our sense are not even unmixed for $n\geq 4$, whereas for $n\leq 3$
they are unmixed. We further give a complete description of their
isolated primary components as the defining ideals of monomial curves
with coefficients. This answers an open question on the number of
primary components of Herzog-Northcott ideals, which comprise the case
$n=3$.  Moreover, we find an explicit, concrete description of the
irredundant embedded component (for $n\geq 4$) and characterize when
the hull of the ideal, i.e., the intersection of its isolated primary
components, is prime. Note that these last results are independent of
the characteristic of the ground field. Our techniques involve the
Eisenbud-Sturmfels theory of binomial ideals and Laurent polynomial
rings, together with theory of Smith Normal Form and of Fitting
ideals. This gives a more transparent and completely general approach,
replacing the theory of multiplicities used previously to treat the
particular case $n=3$.
\end{abstract}

\maketitle

\section{Introduction}\label{introduction}

Let $A=k[x_1,\ldots ,x_n]$ be the polynomial ring in $n$ variables
$\underline{x}=x_1,\ldots ,x_n$ over a field $k$, $n\geq 2$. Set
$\fm=(\underline{x})$, the maximal ideal generated by the $x_i$. Let
$a_{i,j}\in\bn$, $i,j=1,\ldots ,n$, with $a_{i,i}=\sum_{j\neq
  i}a_{i,j}$, and let $L$ be the $n\times n$ integer matrix defined as
follows:
\begin{eqnarray*}
L=\left(\begin{array}{rrcr}a_{1,1}&-a_{1,2}&\ldots &-a_{1,n}\\
-a_{2,1}&a_{2,2}&\ldots &-a_{2,n}\\\vdots&\vdots&\ldots&\vdots\\
-a_{n,1}&-a_{n,2}&\ldots &a_{n,n}\end{array}\right),
\end{eqnarray*}
where the sum of the entries of each row is zero. (Here $\bn$ denotes
the set of positive integers.) For ease of reference, $L$ will be
called a {\em positive critical binomial matrix} (PCB matrix, for
short). Set $d\in\bn$ to be the greatest common divisor of the
$(n-1)\times (n-1)$ minors of $L$. (We shall see below that these
minors are non-zero.)  Let $\underline{f}=f_{1},\ldots,f_{n}$ be the
binomials defined by the columns of $L$:
\begin{eqnarray*}
f_1=x_1^{a_{1,1}}-x_2^{a_{2,1}}\cdots x_n^{a_{n,1}}\mbox{, }
f_2=x_2^{a_{2,2}}-x_1^{a_{1,2}}x_3^{a_{3,2}}\cdots x_n^{a_{n,2}}
\mbox{,}\ldots\mbox{,} f_n=x_n^{a_{n,n}}-x_1^{a_{1,n}}\cdots
x_{n-1}^{a_{n-1,n}}.
\end{eqnarray*}
Let $I=(\underline{f})$ be the binomial ideal generated by the
$f_j$. We will call $I$ the {\em positive critical binomial ideal} (PCB
ideal, for short) associated to $L$.

The purpose of this paper is to investigate the primary decomposition
of PCB ideals and to contrast this theory with analogous results
in \cite{op2} concerning ideals of Herzog-Northcott type, which
comprise the case $n=3$. We first prove that, if $n\geq 4$
(respectively, $n\leq 3$), $I$ has at most $d+1$ (respectively, $d$)
primary components. This answers a question posed
in \cite[Remark~8.6]{op2}.

We will observe that $I$ is contained in a unique toric ideal $\fp_m$
associated to the monomial curve $\Gamma_m=\{(\lambda^{m_1},\ldots
,\lambda^{m_n})\in\ba^{n}_{k}\mid \lambda\in k\}$, where
$m=(m_1,\ldots ,m_n)=m(I)\in\bn^{n}$ is determined by $I$.  That is,
$\fp_m$ (referred to as the {\em Herzog} ideal associated to $m$) is
the kernel of the natural homomorphism $A\to k[t]$, $t$ a variable
over $k$, that sends each $x_i$ to $t^{m_i}$.

In somewhat more detail, if $k$ contains the $d$-th roots of unity and
the characteristic of $k$, $\ch(k)$, is zero or $\ch(k)=p$, $p$ a
prime with $p\nmid d$, we give a full description of a minimal primary
decomposition of $I$. Namely, the intersection of the isolated primary
components of $I$, $\hull(I)$, is equal to the intersection of $d$
prime toric ideals of ``monomial curves with coefficients'', i.e.,
kernels of natural homomorphisms $A\to k[t]$ that send each $x_i$ to
$\lambda_it^{m_i}$, $\lambda_i\in k$. This will explain the
``intrinsic'' role of the Herzog ideal $\fp_{m(I)}$ among the other
minimal primes of $I$ as the instance where each of the
``coefficients'' $\lambda_i$ equals $1$.

Furthermore, if $n\leq 3$, $I$ is unmixed and $I=\hull(I)$. But if
$n\geq 4$, $I$ has one irredundant embedded $\fm$-primary
component. This provides a very striking contrast between the cases
$n\leq 3$ and $n\geq 4$.  In each case we give a concrete description
of these primary components (cf. Theorems~\ref{embeddedcomp} and
\ref{main}).

We now recall briefly from \cite{op2} some relevant parts of the
theory of ideals of Herzog-Northcott type (or HN ideals, as they are
referred to). The study of HN ideals had their origin in work of
Herzog \cite{herzog} on the defining ideals $\fp_m$ of monomial space
curves $\Gamma_m$, $m\in\bn^3$, $\gcd(m)=1$. The ideals $\fp_m$, which
are Cohen-Macaulay almost complete intersection ideals of height two,
proved useful in work of the authors in settling a long-standing open
question on an aspect of the uniform Artin-Rees property
(cf. \cite{op1}); this work built on the observation that these ideals
$\fp_m$ were a particular case of a class of ideals studied by
Northcott \cite{northcott}.

In \cite{op2} we defined an HN ideal $I$ as the determinantal ideal
generated by the $2\times 2$ minors of a certain matrix. One can
easily check that HN ideals and PCB ideals are two notions that
coincide when $A=k[\underline{x}]$ with $n=3$. In \cite[Definition~7.1
  and Remark~7.2]{op2} we introduced an integer vector
$m(I)=(m_1,m_2,m_3)\in\bn^{3}$ associated to $I$. We showed that $I$
is prime if and only if the greatest common divisor of $m(I)$ is equal
to 1 (\cite[Theorem~7.8]{op2}). Further, using techniques from the
theory of multiplicities, we gave upper bounds for the number of prime
components of $I$ in terms of the $m_i$ and $\gcd(m(I))$. Finally,
using a Jacobian criterion, we showed that $I$ is radical if the
characteristic of $k$ is zero or sufficiently large. More
particularly, in \cite[Remark~8.6]{op2} we posed the open question as
to whether the number of prime components of $I$ was at most
$\gcd(m(I))$.

We now return to these matters using the Eisenbud-Sturmfels theory of
binomial ideals (see \cite{es}), and in particular their investigation
of so-called Laurent binomial ideals, to obtain a detailed positive
answer to this conjecture. The Eisenbud-Sturmfels theory used here
provides a more transparent approach that works for general $n$, and
not just when $n=3$. Note that, since a PCB ideal $I$ is binomial, so
also are its isolated primary components, their intersection
$\hull(I)$ and, for $n\geq 4$, even for a suitable choice for its
$\fm$-primary irredundant embedded component. This approach also
enables us to give an analogous criterion for $\hull(I)$ of a general
PCB ideal $I$ to be prime. Observe that when $n\geq 4$, $I$ cannot be
radical since it is not even unmixed. However, we show that, for
suitable coefficient fields $k$, $\hull(I)$ is radical; as stated
above, recall that when $n\leq 3$, $\hull(I)$ coincides with $I$,
since $I$ is then unmixed.

Notice also that for $n=4$, $I$ is somewhat related to the notion of
an ``ideal generated by a full set of critical binomials'' introduced
by Alc\'antar and Villarreal in \cite[\S3, p.~3039]{av}, although the
definitions have essential differences. 

For further background and recent related work from a similar
perspective to ours, see \cite{waldi}, \cite{eto} and especially
\cite{gastinger}, and also \cite{ojeda} and \cite{ko}.

For an alternative combinatorial approach, see the recent paper
\cite{lv} (and Remark~\ref{lvvsop} below) and the survey papers
\cite{km} and \cite{miller}. Specifically, at the end of Section~3 of
\cite{miller}, a general programme is set out whereby binomial primary
decompositions can be calculated. Substantial difficulties could
present themselves as to how this programme plays out as regards
particular binomial ideals and especially as regards abstractly
defined classes of binomial ideals. Our point of view in the present
paper is to use constructive Commutative Algebra to give explicit,
concrete descriptions of binomial primary decompositions of PCB
ideals. In particular, we present in the case of PCB ideals an
explicit solution to the `problem' mentioned in
\cite[Remark~3.5]{dmm}, in that Theorem~\ref{embeddedcomp} below
provides a concrete description of the single embedded component of an
irredundant binomial primary decomposition of a PCB ideal in the case
$n\geq 4$ where this ideal is not unmixed.

The paper is organised as follows. In Section~\ref{endowing} we first
observe that all the rows of $\adj(L)$, the adjoint matrix of $L$, are
equal and lie in $\bn^n$. Then we define the integer vector
$m(I)\in\bn^{n}$ associated to a PCB binomial $I$ as the last row of
$\adj(L)$ (see Definition~\ref{defm(I)}). We see that this definition
extends the one given in \cite[Definition~7.1 and
  Remark~7.2]{op2}. Moreover, this vector $m(I)$ helps define a
grading on $A$ in which $I$ becomes homogeneous.

In Section~\ref{first}, we recover and extend properties of HN ideals,
namely we show in Propositions~\ref{gradeI} and \ref{aci} that a PCB
ideal $I$ is contained in a unique Herzog ideal, specifically
$\fp_{m(I)}$, and that, if $n\geq 3$, $I$ is an almost complete
intersection. Section~\ref{onthe} is devoted to study the
(un)mixedness property of PCB ideals. The result, stated above, is
surprising: while for $n\leq 3$, $I$ is unmixed and $I=\hull(I)$, for
$n\geq 4$ we find that $I$ is never unmixed (see
Remark~\ref{unmixedifn=3} and Proposition~\ref{mixedifn>3}). We also
provide an explicit, concrete description of $\hull(I)$ and, when
$n\geq 4$ (in which case $I$ is never unmixed), of a choice for the
irredundant embedded component of $I$, each of these descriptions
being independent of the characteristic of $k$
(cf. Proposition~\ref{explicitsi} and
Theorem~\ref{embeddedcomp}). This gives a comprehensive and concrete
solution to \cite[Problem~6.3]{es} in the case of our binomial ideals.

In Section~\ref{ashort}, we review the normal decomposition of an
integer matrix (also called the Smith Normal Form). This will lead, on
the one hand, to a change of variables that will greatly simplify the
description of $I$. On the other hand, it relates the greatest common
divisor of $m(I)$ (i.e., $d$, the greatest common divisor of the
entries of $\adj(L)$) with the cardinality of the torsion group of the
abelian group generated by the columns of $L$
(Proposition~\ref{torsion}). 

In Section~\ref{applying} we pass to the Laurent polynomial ring,
apply the change of variables given by the normal decomposition of
$L$, get a better description of $I$ in the Laurent ring, and then
contract back to the original polynomial ring
(cf. Theorem~\ref{dcompofSI}). This approach also enables us in
Corollary~\ref{S(I)prime} to characterize when $\hull(I)$ is prime. In
turn, we can use Corollary~\ref{S(I)prime} to show that the class of
PCB ideals has minimal overlap with the class of binomial ideals,
namely so-called lattice basis ideals for saturated lattices,
considered by Ho\c{s}ten and Shapiro in \cite{hs} (cf.
Proposition~\ref{PCBvshs}).

Finally, in the last section, we use the expression obtained in
Theorem~\ref{dcompofSI} to prove the main result of the paper:
Theorem~\ref{main}. We end by giving some illustrative examples.

Throughout the paper we will use the following notations:
$A=k[\underline{x}]=k[x_1,\ldots ,x_n]$ will be the polynomial ring in
$n$ variables $\underline{x}=x_1,\ldots ,x_n$ over a field $k$, $n\geq
2$. The maximal ideal generated by $\underline{x}$ will be denoted
$\fm=(\underline{x})$. The multiplicatively closed set in $A$
generated by $x=x_1\cdots x_n$, the product of the variables
$x_1,\ldots ,x_n$, will be denoted by $S$, and
$B=S^{-1}A=k[\underline{x}^{\pm}]=k[x_1,\ldots ,x_n,x_1^{-1},\ldots
  ,x_n^{-1}]$ will be the corresponding {\em Laurent polynomial ring}.

We will use the following multi-index notation: for
$\alpha=(\alpha_1,\ldots ,\alpha_n)\in\bz^{n}$, or more generally,
$\alpha$ a row or a column of a matrix with ordered entries
$\alpha_1,\ldots ,\alpha_n\in\bz$, set
$x^{\alpha}=x_1^{\alpha_1}\cdots x_n^{\alpha_{n}}$ in $B$.  Given such
an $\alpha=(\alpha_1,\ldots ,\alpha_n)\in\bz^{n}$, let
$\alpha_{+}=\max(\alpha,0)\in \bn_0$ and $\alpha_{-}=-\min(\alpha,0)\in
\bn_0$, where $\bn_0:=\bn\cup\{0\}$, so that
$\alpha=\alpha_{+}-\alpha_{-}$.

By a binomial in $A$ we understand a polynomial of $A$ with at most
two terms, say $\lambda x^\alpha-\mu x^\beta$, where $\lambda, \mu\in
k$ and $\alpha, \beta\in\bn^{n}_{0}$. A binomial ideal of $A$ is an
ideal of $A$ generated by binomials.

Unless stated otherwise, $L$ will always be a PCB matrix, i.e., an
$n\times n$ integer matrix defined as above,
$\underline{f}=f_{1},\ldots,f_{n}$ will be the binomials defined by
the columns of $L$ and $I=(\underline{f})$ will the PCB ideal of $A$
associated to $L$.

Given an $n\times s$ integer matrix $M$, we will denote by $m_{i,*}$
and $m_{*,j}$ its $i$-th row and $j$-th column, respectively. Then
$f_{m_{*,j}}=x^{(m_{*,j})_{+}}-x^{(m_{*,j})_{-}}$ will denote the
binomial defined by the $j$-th column of $M$. The ideal
$I(M)=(f_{m_{*,j}}\mid j=1,\ldots ,s)$, generated by the binomials
$f_{m_{*,j}}$, will be called the binomial ideal associated to the
matrix $M$. For instance, a PCB ideal $I$ is the binomial ideal
$I=I(L)$ associated to a PCB matrix $L$.

For an $n\times s$ integer matrix $M$, we will denote by $\bm\subset
\bz^n$ to the subgroup spanned by the columns of $M$ ($\bm$ is often
called a lattice of $\bz^n$). In other words, $\bm=\bz m_{*,1}+\ldots
+\bz m_{*,s}=\varphi(\bz^s)\subseteq \bz^n$, where
$\varphi:\bz^{s}\to\bz^{n}$ is the homomorphism defined by the matrix
$M$. The binomial ideal $I(\bm)=(x^{m_{+}}-x^{m_{-}}\mid m\in \bm)$ is
usually called the lattice ideal of $A$ associated to $\bm$ (see,
e.g., \cite[Definition~7.2]{ms} or alternatively \cite[just before
  Corollary~2.5]{es}, where $I(\bm)$ is denoted by $I_{+}(\rho)$,
$\rho:\bm\to k^{*}$ being the trivial partial character on the lattice
$\bm$; see also \cite[Corollary~7.1.4]{villarreal}).

By an $n\times n$ invertible integer matrix, we will
understand an $n\times n$ matrix $P$ with entries in $\bz$ whose
determinant is $\pm 1$. Thus its inverse matrix $P^{-1}$ is also an
integer matrix.

\section{Endowing $A$ with a grading that makes $I$ 
homogeneous}\label{endowing}

Let $L_{i,j}$ be the $(i,j)$-cofactor of an $n\times n$ PCB matrix
$L$, i.e., the $(n-1)\times (n-1)$ matrix obtained from $L$ by
eliminating the $i$-th row and the $j$-th column of $L$. Let
$h_{i,j}=(-1)^{i+j}\det(L_{j,i})$ and set $H=(h_{i,j})=\adj(L)$, the
adjoint matrix of $L$. In the next result, all the computations are
thought of in $\bz$ or $\bq$ (i.e., in characteristic zero) and the
ranks are taken over $\bq$.

\begin{lemma}\label{rankofL} 
With the notations above:
\begin{itemize}
\item[$(a)$] $\det(L_{i,i})>0$, for all $i=1,\ldots ,n$. In
  particular, $\rank(L)=n-1$;
\item[$(b)$] $\det(L_{i,n})=(-1)^{n-i}\det(L_{i,i})$, for all
  $i=1,\ldots ,n$;
\item[$(c)$] $\det(L_{i,j})=(-1)^{n-j}\det(L_{i,n})$, for all
  $i,j=1,\ldots ,n$.
\end{itemize}
Moreover,
\begin{itemize}
\item[$(d)$] $h_{i,j}>0$, for all $i,j=1,\ldots ,n$;
\item[$(e)$] $h_{i,*}=h_{n,*}$, for all $i=1,\ldots,n$. In particular
  $\rank(\adj(L))=1$;
\item[$(f)$] $\nulls(L^{\top})$ is generated as a $\bq$-linear
  subspace by $h_{n,*}^{\top}$, the transpose of the last row of
  $\adj(L)$.
\end{itemize}
\end{lemma}
\begin{proof} 
The proof of $(a)$ follows easily from standard facts about so-called
strictly diagonally dominant matrices (cf., e.g., an easy adaptation
of the statement and proof of \cite[Bemerkung 6.1;
pp.~37-38]{gastinger} where one employs induction based on the number
of rows, using row reduction). We present here another proof based on
a general fact about the eigenvalues of such matrices.  Fix $i\in
\{1,\ldots ,n\}$.  By the Gershgorin Circle Theorem, every (possibly
complex) eigenvalue $\lambda$ of $L_{i,i}$ lies within at least one of
the discs $\{z\in\bc\mid |z-a_{j,j}|\leq R_{j}\}$, $j\neq i$, where
$R_j=\sum_{u\neq i,j}|-a_{j,u}|<a_{j,j}$ since $L_{i,i}$ is a strictly
diagonally dominant matrix. If $\lambda\in\br$, then $\lambda>0$. If
$\lambda\not\in\br$, then since $L_{i,i}$ is a real matrix, its conjugate
$\overline{\lambda}$ must also be an eigenvalue of $L_{i,i}$. By means
of the Jordan canonical form of $L_{i,i}$, one deduces that
$\det(L_{i,i})>0$. Clearly $(1,\ldots ,1)^{\top}$ is in the nullspace
of $L$ and $(a)$ holds.

Fix $i\in \{1,\ldots ,n-1\}$. By performing $n-1-i$ permutations, the
$i$-th column of $L_{i,n}$ may be taken to the outer right hand
side. Add to this new right hand column the sum of the other columns
and change the sign. Using that the sum of the entries of each row of
$L$ is zero, one gets in the outer right hand column the $n$-th column
of $L_{i,i}$.  Therefore $\det(L_{i,n})=(-1)^{n-i}\det(L_{i,i})$. This
proves $(b)$.

Let $j\in \{1,\ldots ,n-1\}$. Since the sum of the entries of each row
of $L$ is zero, to calculate $\det(L_{i,j})$ one can substitute the
last column of $L_{i,j}$ by the corresponding $j$-th column of
$L_{i,n}$ with the sign changed. By performing $n-1-j$ permutations,
one gets the matrix $L_{i,n}$. Therefore
$\det(L_{i,j})=(-1)^{n-j}\det(L_{i,n})$. This proves $(c)$.

For $i,j=1,\ldots,n$, using $(c)$, $(b)$ and $(a)$, respectively, we
have
\begin{eqnarray*}
h_{i,j}=(-1)^{i+j}\det(L_{j,i})=(-1)^{i+j+n-i}\det(L_{j,n})=
(-1)^{i+j+n-i+n-j}\det(L_{j,j})=\det(L_{j,j}). 
\end{eqnarray*}
This proves $(d)$.

For $j\in \{1,\ldots ,n-1\}$, using $(c)$,
\begin{eqnarray*}
h_{i,j}=(-1)^{i+j}\det(L_{j,i})=(-1)^{i+j}(-1)^{n-i}\det(L_{j,n})=
(-1)^{n+j}\det(L_{j,n})=h_{n,j}. 
\end{eqnarray*}
Therefore all the rows of $\adj(L)$
are equal. In particular, since $\adj(L)\neq 0$ by $(d)$, we see that
$\rank(\adj(L))=1$. This proves $(e)$.

Since $\rank(L^{\top})=\rank(L)=n-1$, we have that
$\dim\nulls(L^{\top})=1$. Furthermore, since $\adj(L)L=0$, the
transpose of the (non-zero) last row of $\adj(L)$ generates the
$\bq$-linear subspace $\nulls(L^{\top})$.
\end{proof}

As before, set $A=k[\underline{x}]$, the polynomial ring in $n$
variables $\underline{x}=x_1,\ldots ,x_n$ over a field $k$, $n\geq 2$.

\begin{definition}\label{defm(I)}
{\rm Let $I=(\underline{f})$ be the PCB ideal associated to $L$. Let
  $m=m(I)=(m_1,\ldots,m_n)$ be the $n$-th row of $\adj(L)$; this will
  be called the {\em integer vector associated to} $I$. By the
  previous lemma, $m(I)\in\bn^{n}$ and $m(I)^{\top}$ is a basis of the
  $\bq$-linear subspace $\nulls(L^{\top})$. We will denote by $d$ the
  greatest common divisor of the coefficients of $m(I)$,
  $d:=\gcd(m(I))$, and set $\nu(I)=m(I)/d=(\nu_1,\ldots
  ,\nu_n)\in\bn^{n}$. From now on, given a PCB ideal $I$ of $A$, we
  will endow $A$ with the natural grading induced by giving $x_{i}$
  weight $\nu_{i}$. Then $A$ is graded by $\bn_{0}:=\bn\cup\{0\}$, and
  $\underline{x}$ and $\underline{f}$ are homogeneous elements of
  positive degree. In particular, $I$ is homogeneous. Hence so are its
  isolated primary components and its associated primes, and an
  irredundant embedded primary component may be chosen homogeneous
  (see, e.g., \cite[Ch.~VII,\S~2, Theorem~9 and Corollary,
  pp.153-154]{zs}).  }\end{definition}

\begin{remark}\label{casen=2}
{\rm For $n=2$ we have $I=(f_1,f_2)$, where
  $f_1=x_1^{a_{1,1}}-x_2^{a_{2,1}}$ and
  $f_2=x_2^{a_{2,2}}-x_1^{a_{1,2}}$, with $a_{1,1}=a_{1,2}$ and
  $a_{2,2}=a_{2,1}$. Thus $f_2=-f_1$ and $I=(f_1)$ is a complete
  intersection. In particular, $I$ is unmixed. Here,
  $m(I)=(a_{2,2},a_{1,1})\in\bn^{2}$ and
  $d=\gcd(m(I))=\gcd(a_{1,1},a_{2,2})$.  }\end{remark}

\begin{remark}\label{casen=3}
{\rm For $n=3$ we have $I=(f_1,f_2,f_3)$, where
  $f_1=x_1^{a_{1,1}}-x_2^{a_{2,1}}x_3^{a_{3,1}}$,
  $f_2=x_2^{a_{2,2}}-x_1^{a_{1,2}}x_3^{a_{3,2}}$ and
  $f_3=x_3^{a_{3,3}}-x_1^{a_{1,3}}x_2^{a_{2,3}}$, with
  $a_{1,1}=a_{1,2}+a_{1,3}$, $a_{2,2}=a_{2,1}+a_{2,3}$ and
  $a_{3,3}=a_{3,1}+a_{3,2}$. Observe that $f_1,f_2,f_3$ are, up to
  sign, the $2\times 2$ minors of the matrix
\begin{eqnarray*}
\left(\begin{array}{ccc}x_1^{a_{1,2}}&x_2^{a_{2,3}}&x_3^{a_{3,1}}
  \\x_2^{a_{2,1}}&x_3^{a_{3,2}}&x_1^{a_{1,3}}\end{array}\right).
\end{eqnarray*}
It follows that when $n=3$, PCB ideals are precisely the ideals of
Herzog-Northcott type, or HN ideals for short, considered in
\cite{op2}. In the proof of \cite[Remark~4.4]{op2}, there appear
positive integers $m_1,m_2,m_3$ presented as the $2\times 2$ minors of
the matrix defining the exponents of $f_1$ and $f_2$. Subsequently, in
\cite[Definition~7.1 and Remark~7.2]{op2}, $(m_1,m_2,m_3)$ is defined
as the integer vector associated to the Herzog-Northcott ideal $I$. In
conclusion, one can easily check that, when $n=3$, the present
definition of $m(I)$ coincides with the one given in
\cite[Definition~7.1 and Remark~7.2]{op2}.  }\end{remark}

\begin{remark}\label{casen=4}
{\rm It is a long-standing open problem to find a minimal generating
  set for the defining ideals $\fp_m$ of monomial curves $\Gamma_m$,
  $m\in\bn^n$, $\gcd(m)=1$, and to decide whether the $\fp_m$ are set
  theoretically complete intersections. For $n=3$, the problem was
  completely solved by Herzog in \cite{herzog}. For $n=4$, and
  provided that $\fp_m$ is an almost complete intersection, Gastinger
  in \cite{gastinger} and Eto in \cite{eto} gave a definitive
  answer. In an attempt to study this problem for $n=4$, Alc\'antar
  and Villarreal defined in \cite{av} what they called a {\em full set
    of critical binomials} as a set of four binomials
  $f_1,f_2,f_3,f_4\in\fp_m$, where $m=(m_1,m_2,m_3,m_4)\in\bn^4$,
  $m_1<m_2<m_3<m_4$ and $\gcd(m)=1$. The $f_i$ were respectively
  defined as in our introduction, namely
\begin{eqnarray*}
x_1^{a_{1,1}}-x_2^{a_{2,1}}x_3^{a_{3,1}}x_4^{a_{4,1}},
x_2^{a_{2,2}}-x_1^{a_{1,2}}x_3^{a_{3,2}}x_4^{a_{4,2}},
x_3^{a_{3,3}}-x_1^{a_{1,3}}x_2^{a_{2,3}}x_4^{a_{4,3}},
x_4^{a_{4,4}}-x_1^{a_{1,4}}x_2^{a_{2,4}}x_3^{a_{3,4}},
\end{eqnarray*}
but with $a_{i,i}>0$ and $a_{i,j}\in\bn_0$, and such that $a_{i,i}$ is
minimal with respect to the condition $a_{i,i}m_i\in \sum_{j\neq
  i}m_j\bn_0$. They then studied when the ideal generated by the $f_i$
is the whole of $\fp_m$. As is clear, our definition of PCB ideal for
$n=4$ does not exactly match their definition. On the one hand, we do
not allow zero exponents, and on the other hand we do not ask for
the above minimal condition or for restrictions on the $m_i$.
}\end{remark}

\section{First properties of PCB ideals}\label{first}

Set $A=k[\underline{x}]$ to be the polynomial ring in $n$ variables
$\underline{x}=x_1,\ldots ,x_n$ over a field $k$, $n\geq 2$. We start
this section by recovering a definition from \cite{op2}.

\begin{definition}\label{herzogideal}
{\rm Let $u=(u_{1},\ldots ,u_{n})\in\bn^{n}$ be an integer vector with
  greatest common divisor not necessarily equal to 1. The {\em Herzog
    ideal associated to $u$} is the prime ideal $\fp_{u}$ defined as
  the kernel of the natural homomorphism $\varphi_{u}:A\rightarrow
  k[t]$ that sends $x_{i}$ to $t^{u_{i}}$, for each $i=1,\ldots ,n$.
}\end{definition}

The following is a list of well-known properties of Herzog ideals,
with a sketched proof for the sake of completeness.

\begin{remark}\label{herzogproperties} {\rm 
Let $u\in\bn^{n}$. The extension $k[t^{u_1},\ldots ,t^{u_n}]\subset
k[t]$ is integral. Hence $A/\fp_{u}\cong k[t^{u_1},\ldots,t^{u_n}]$
has Krull dimension $1$ and $\fp_{u}$ is a prime ideal of height
$n-1$. Since $0\in V(\fp_u)\subseteq \ba^{n}_k$, where $V(\fp_u)$
denotes the affine set of zeros over $k$ of $\fp_u$,
$\fm=I(\{0\})\supseteq I(V(\fp_u))\supseteq \fp_u$ and
$\fp_u\subsetneq\fm$. Moreover, if $v\in\bn^{n}$ is such that $u=dv$
for some $d\in \bn$, clearly $\fp_u\supseteq \fp_v$ and, by the
equality of heights, $\fp_{u}=\fp_{v}$.

We claim that if $\gcd(u)=1$, then $V(\fp_u)=\Gamma_u:=
\{(\lambda^{u_1},\ldots ,\lambda^{u_n})\in\ba^{n}_{k}\mid \lambda\in
k\}$ (see \cite[Proposition~2.9]{rvz}). Clearly
$\Gamma_u\subseteq V(\fp_u)$. Note that for $i=2,\ldots,n$,
$x_1^{u_i}-x_i^{u_1}$ is in $\fp_u$.  Hence if
$(\lambda_1,\ldots,\lambda_n)\in V(\fp_u)\setminus\{0\}$, then each
$\lambda_i\neq 0$ and, taking $\lambda:=\lambda_1^{c_1}\cdots
\lambda_n^{c_n}$ where $c_1u_1+\ldots +c_nu_n=1$ with $c_i\in\bz$, one
has $\lambda^{u_i}=\lambda_i$ and hence
$(\lambda_1,\ldots,\lambda_n)\in \Gamma_u$.

Moreover, $\fp_u=(x_1-t^{u_1},\ldots ,x_n-t^{u_n})\cap A$, where the
ideal $(x_1-t^{u_1},\ldots ,x_n-t^{u_n})$ is considered in
$A[t]=k[x_1,\ldots ,x_n,t]$. Indeed, if $f\in\fp_u$,
\begin{eqnarray*}
&&f=\sum a_{\alpha}x^{\alpha}=\sum
  a_{\alpha}(x_1-t^{u_1}+t^{u_1})^{\alpha_1}\cdots
  (x_n-t^{u_n}+t^{u_n})^{\alpha_n}=\\ && g+\sum
  a_{\alpha}(t^{u_1})^{\alpha_1}\cdots
  (t^{u_n})^{\alpha_n}=g+f(t^{u_1},\ldots ,t^{u_n})=g+\varphi_u(f)=g,
\end{eqnarray*}
where $g\in (x_1-t^{u_1},\ldots ,x_n-t^{u_n})$. Thus, $f=g\in
(x_1-t^{u_1},\ldots ,x_n-t^{u_n})\cap A$. The other inclusion follows
easily. In particular, by \cite[Corollary~1.3]{es}, $\fp_u$ is a
binomial ideal.

Finally, if $k$ is infinite and $\gcd(u)=1$, we claim that
$\fp_u=I(\Gamma_u)$, the vanishing ideal of $\Gamma_u$. On the one
hand, since $\Gamma_u=V(\fp_u)$, $I(\Gamma_u)=I(V(\fp_u))\supseteq
\fp_u$.  On the other hand, let $f\in I(\Gamma_u)\subset A\subset
A[t]$. The argument above shows that
$f(\underline{x})=g(\underline{x},t)+r(t)$, with $g\in
(x_1-t^{u_1},\ldots ,x_n-t^{u_n})\subset A[t]$ and $r\in k[t]$. For
any $\lambda\in k$, evaluate $x_i$ in $\lambda^{u_i}$ and $t$ in
$\lambda$. Then $0=f(\lambda^{u_1},\ldots
,\lambda^{u_n})=g(\lambda^{u_1},\ldots ,\lambda^{u_n},\lambda )+
r(\lambda)=r(\lambda)$. Thus $r(\lambda)=0$ for all $\lambda\in
k$. Since $k$ is infinite, $r=0$ and
$f(\underline{x})=g(\underline{x},t)\in (x_1-t^{u_1},\ldots
,x_n-t^{u_n})\cap A=\fp_u$.  }\end{remark}

The next result gives us the first properties of a PCB ideal.

\begin{proposition}\label{gradeI}
Let $I=(\underline{f})$ be the PCB ideal associated to $L$. Then the
following hold.
\begin{itemize}
\item[$(a)$] Any subset of $n-1$ elements of $\underline{f}$ is a
  regular sequence in $A$.
\item[$(b)$] $\fp_{m(I)}$ is the unique Herzog ideal containing $I$
  and is a minimal prime over $I$. In particular, $\height(I)=n-1$.
\item[$(c)$] If $n=2$, $I$ is principal. If $n\geq 3$, $f_1,\ldots
  ,f_n$ is a minimal (homogeneous) system of generators of $I$ and
  every (non-necessarily homogeneous) system of generators of $I$ has
  at least $n$ elements.
\end{itemize}
\end{proposition}
\begin{proof}
Since $(f_1,\ldots ,f_{n-1},x_n)=(x_1^{a_{1,1}},\ldots
,x_{n-1}^{a_{n-1,n-1}},x_n)$, the grades of these ideals are equal and
coincide with $\grade(x_1,\ldots ,x_n)=n$ (see, e.g.,
\cite[Exercise~3.1.12(c)]{kaplansky}). Using that $A$ is graded and
that $f_1,\ldots ,f_{n-1},x_n$ are homogeneous, we deduce that these
elements form a regular sequence in any order (see, e.g.,
\cite[Theorem~4.1]{op2}) (and similarly for the possible variations on
this argument). This proves $(a)$.

Given $v\in\bn^{n}$, clearly $I\subseteq \fp_{v}$ if and only if $v$
satisfies the system of equations $vL=0$, i.e., if and only if
$v^{\top}$ is in the nullspace of $L^{\top}$, which by
Lemma~\ref{rankofL} and Definition~\ref{defm(I)} is the $\bq$-linear
subspace generated by $m(I)^{\top}$. Therefore $I\subseteq
\fp_{m(I)}$.  Since $n-1\leq \grade(I)=\height(I)\leq
\height(\fp_{m(I)})=n-1$, $\fp_{m(I)}$ is a minimal prime over $I$ and
$\height(I)=n-1$.  On the other hand, if $I\subseteq \fp_{v}$, for
some $v\in\bn^{n}$, then $vL=0$ and $rv=sm(I)$, with
$r,s\in\bn$. Hence $\fp_{v}=\fp_{rv}=\fp_{sm(I)}=\fp_{m(I)}$.

Suppose that $n\geq 3$. We see first that $f_1,\ldots ,f_n$ is a
minimal homogeneous system of generators of $I$ in the sense that none
of them is irredundant. For, if $f_n$ were redundant, say, since
$n\geq 3$, $I=(f_1,\ldots ,f_{n-1})\subseteq (x_1,\ldots ,x_{n-1})$
and $f_n=g_1x_1+\ldots +g_{n-1}x_{n-1}$, for some $g_i\in
A$. Substituting $x_1,\ldots, x_{n-1}$ by 0 and $x_{n}$ by 1, one
would get a contradiction. By \cite[Proposition~1.5.15]{bh}, every
minimal homogeneous system of generators of $I$ has exactly
$\mu(I_{\mathfrak{m}})$ elements. Hence
$n=\mu(I_{\mathfrak{m}})$. Finally, if $h_1,\ldots ,h_r$ is a minimal
(non-necessarily homogeneous) system of generators of $I$, $h_1,\ldots
,h_r$ certainly lie in $\fm$, and $h_1,\ldots ,h_r$ in
$A_{\mathfrak{m}}$ still generate $I_{\mathfrak{m}}$. Thus
$r\geq\mu(I_{\mathfrak{m}})=n$.
\end{proof}

\begin{remark}\label{explicitrelation}
{\rm Similarly to \cite[Remark~6.2]{op2}, we can show a relation among
  $f_1,\ldots ,f_n$. Concretely, $x^{b(1)}f_1+\ldots
  +x^{b(n)}f_{n}=0$, where the $b(i)\in\bn_{0}^{n}$ are defined as
  follows:
\begin{eqnarray*}
&&b(1)=(0,0,a_{3,3}-a_{3,4}-\ldots -a_{3,n}-a_{3,1},
  a_{4,4}-a_{4,5}-\ldots -a_{4,n}-a_{4,1},\ldots ,a_{n,n}-a_{n,1}),
  \\ &&b(2)=(a_{1,1}-a_{1,2},0,0,a_{4,4}-a_{4,5}-\ldots
  -a_{4,n}-a_{4,1}-a_{4,2},\ldots , a_{n,n}-a_{n,1}-a_{n,2}),
  \\&&b(3)=(a_{1,1}-a_{1,2}-a_{1,3},a_{2,2}-a_{2,3},0,0,\ldots
  ,a_{n,n}-a_{n,1}-a_{n,2}-a_{n,3})\mbox{, }\ldots ,\\ &&
  b(n-1)=(a_{1,1}-a_{1,2}-\ldots -a_{1,n-1},\ldots 
  ,a_{n-2,n-2}-a_{n-2,n-1},0,0)\mbox{ and } \\&&
  b(n)=(0,a_{2,2}-a_{2,3}-\ldots -a_{2,n},a_{3,3}-a_{3,4}-\ldots
  -a_{3,n},\ldots ,a_{n-1,n-1}-a_{n-1,n},0).
\end{eqnarray*}
For instance, when $n=2$, $b(1)=b(2)=(0,0)$ and
$x^{b(1)}f_1+x^{b(2)}f_2=f_1+f_2$, which is certainly zero. For $n=3$,
since the sum of the entries of each row is zero,
$b(1)=(0,0,a_{3,2})$, $b(2)=(a_{1,3},0,0)$ and
$b(3)=(0,a_{2,1},0)$. Thus $x^{b(1)}f_1+x^{b(2)}f_2+x^{b(3)}f_3=
x_3^{a_{3,2}}f_1+x_1^{a_{1,3}}f_2+x_2^{a_{2,1}}f_3=0$, which is (up to
sign) the second syzygy in \cite[Remark~6.2]{op2}. For $n=4$, we have
\begin{eqnarray*}
x_3^{a_{3,2}}x_4^{a_{4,2}+a_{4,3}}f_1+x_4^{a_{4,3}}x_1^{a_{1,3}+a_{1,4}}f_2+
x_1^{a_{1,4}}x_2^{a_{2,4}+a_{2,1}}f_3+x_2^{a_{2,1}}x_3^{a_{3,1}+a_{3,2}}f_4=0.
\end{eqnarray*}
}\end{remark}

With respect to the property of being an almost complete intersection
(in the sense of Herrmann, Moonen and Villamayor \cite{hmv}), we have
a result similar to that of \cite[Proposition~6.3]{op2}.

\begin{proposition}\label{aci}
Let $I=(\underline{f})$ be the PCB ideal associated to $L$. Then the
following hold.
\begin{itemize}
\item[$(a)$] For any associated prime $\fp$ of $I$, either
  $\height(\fp)=n-1$ and $x_{i}\not\in\fp$, for all $i=1,\ldots ,n$,
  or else $\fp=\fm$.
\item[$(b)$] For any minimal prime ideal $\fp$ over $I$,
  $IA_{\mathfrak{p}}$ is a complete intersection.
\item[$(c)$] If $n=2$, $I$ is a complete intersection. If $n\geq 3$,
  $I$ is an almost complete intersection.
\end{itemize}
\end{proposition}
\begin{proof}
Let $\fp$ an associated prime of $I$. Since $I$ is homogeneous, $\fp$
is homogeneous too and hence $\fp\subseteq \fm$ (see, e.g.,
\cite[\S~1.5]{bh}). If $\fp\subsetneq \fm$, since $\height(I)=n-1$,
then $\height(\fp)=n-1$ too. Moreover, for each $i$, $x_i\not\in\fp$,
otherwise $(\underline{f},x_i)\subseteq \fp$ and $\fp=\fm$.

Let $\fp$ be a minimal prime over $I$, so in particular $\fp\neq \fm$
(because $I\subseteq \fp_{m(I)}\subsetneq \fm$). Thus $x_i\not\in\fp$,
for all $i=1,\ldots ,n$. Using Remark~\ref{explicitrelation}, and with
$x=x_1\cdots x_n$ as before, $IA_{x}=(f_1,\ldots ,f_{n-1})A_{x}$ and
$IA_{\mathfrak{p}}=(IA_{x})A_{\mathfrak{p}}=(f_{1},\ldots
,f_{n-1})A_{\mathfrak{p}}$, where $f_{1},\ldots ,f_{n-1}$ is a regular
sequence in $A_{\mathfrak{p}}$.

Finally, if $n=2$, $I$ is a complete intersection (cf.
Remark~\ref{casen=2}). If $n\geq 3$, by Proposition~\ref{gradeI}$(a)$,
$(c)$, $I$ has height $n-1$ and is minimally generated by $n$
elements. Since $I$ is locally a complete intersection at its minimal
primes, $I$ is an almost complete intersection.
\end{proof}

\section{On the (un)mixedness property of PCB ideals}\label{onthe}

Let $S$ be the multiplicatively closed set in $A=k[\underline{x}]$
generated by $x=x_1\cdots x_n$. Let
$B=S^{-1}A=k[\underline{x}^{\pm}]=k[x_1,\ldots ,x_n,x_1^{-1},\ldots
  ,x_n^{-1}]$ be the Laurent polynomial ring. As usual, if $I$ is an
ideal of $A$, $IB$ will denote its extension in $B$, and, if $J$ is an
ideal of $B$, $J\cap A=J^{c}$ will denote its contraction in $A$. We
will also use the notation $S(I)=IB\cap A$ for the contraction of the
extension of an ideal $I$ of $A$.

Following the notation in \cite[p.~31]{es}, we write $\hull(I)$ for the
intersection of the isolated primary components of $I$.

Note that, if $\alpha\in\bn^{n}_0$, and according to our multi-index
notation, $x^{\alpha}$ is not normally a power of $x=x_1\cdots x_n$
but rather is a monomial in $x_1,\ldots, x_n$. This monomial
$x^{\alpha}$ is indeed a unit in the localized ring $A_x$, since $A_x$
equals the Laurent ring
$B=A[x_1,\ldots,x_n,x_1^{-1},\ldots,x_n^{-1}]$.

The next (standard) result helps to describe the associated primes of
$I$ in terms of the associated primes of $IB$, its extension in $B$.

\begin{proposition}\label{S(I)}
Let $I$ be a PCB ideal of $A$. Then the following hold.
\begin{itemize}
\item[$(a)$] $S(I)=\hull(I)$.
\item[$(b)$] Either $I$ is unmixed and $I=S(I)$, or else $I=S(I)\cap
  \fc$, where $\fc$ is $\fm$-primary and this intersection is irredundant.
\item[$(c)$] If $\alpha\in\bn_0^n\setminus \{0\}$,
  $S(I)=I:(x^{\alpha})^{\infty}:= \{f\in A\mid fx^{N\alpha}\in I,
  \mbox{ for some }N\gg 0\}$.
\item[$(d)$] Suppose that $IB=\fb_1\cap \ldots \cap \fb_r$ is a
  minimal primary decomposition of $IB$ in $B$. Then
  $S(I)=\fb_1^c\cap \ldots \cap \fb_r^c$ is a minimal primary
  decomposition of $S(I)$ in $A$ and $\rad(\fb_i^c)=\rad(\fb_i)^c$.
\end{itemize}
\end{proposition}
\begin{proof}
By Proposition~\ref{aci}$(a)$, $I$ has a minimal primary decomposition
either of the form $I=\fa_1\cap\ldots\cap\fa_r$, or else $I=\fa_1\cap
\ldots \cap\fa_r\cap \fc$, where the $\fa_j$ are $\fp_j$-primary ideals
with $\height(\fp_j)=n-1$, and $\fc$ is $\fm$-primary. In particular,
$x_i\not\in\fp_j$ for each $i$. Therefore $IB=\fa_1B\cap\ldots\cap
\fa_rB$ is a minimal primary decomposition of $IB$ in $B$ and
$S(I)=\fa_1\cap\ldots \cap\fa_r$, which is precisely equal to
$\hull(I)$.

Moreover, either $I$ is unmixed and $I=S(I)$, or else $I=S(I)\cap \fc$,
where $\fc$ is $\fm$-primary and this intersection is irredundant. This
proves $(b)$.

If $I=\fa_1\cap\ldots\cap\fa_r$, $\alpha\in\bn^n_0\setminus \{0\}$ and
$N\gg 0$,
$I:x^{N\alpha}=\cap_{j=1}^{r}(\fa_j:x^{N\alpha})=\cap_{j=1}^{r}\fa_j$,
because for all $i,j$, $x_i\not\in\fp_j=\rad(\fa_j)$ and $\fa_j$ is
$\fp_j$-primary. On the other hand, if $I=\fa_1\cap \ldots
\cap\fa_r\cap \fc$ and $N\gg 0$, then
$I:x^{N\alpha}=\cap_{j=1}^{r}(\fa_j:x^{N\alpha})\cap
(\fc:x^{N\alpha})=\cap_{j=1}^{r}\fa_j$ again, because $\rad(\fc)=\fm$ and
$\fc:x^{N\alpha}=A$, for $N\gg 0$. Thus, in both cases,
$I:x^{N\alpha}=\fa_1\cap\ldots\cap\fa_r=S(I)$ when $N\gg 0$.

Finally, if $IB=\fb_1\cap \ldots \cap \fb_r$ is a minimal primary
decomposition of $IB$ in $B$, then $S(I)=\fb_1^c\cap\ldots \cap
\fb_r^c$ is a primary decomposition of $S(I)$ in $A$, where
$\rad(\fb_i^c)=\rad(\fb_i)^c$. Moreover, if $\fb_1^c\supseteq
\fb_2^c\cap \ldots\cap\fb_r^c$, say, then, since $S^{-1}A$ is a flat
extension of $A$, $\fb_1=\fb_1^{ce}\supseteq \fb_2^{ce}\cap
\ldots\cap\fb_r^{ce}=\fb_2\cap \ldots\cap\fb_r$, a
contradiction. Therefore $S(I)=\fb_1^c\cap\ldots \cap \fb_r^c$ is a
minimal primary decomposition.
\end{proof}

Before proceeding we state, for the sake of reference, a list of
well-known properties of lattice ideals.

\begin{proposition}\label{bm} 
Let $M$ be an $n\times s$ integer matrix and let $\bm\subseteq \bz^n$
be the lattice spanned by the columns of $M$. Let
$I(M)=(x^{(m_{*,j})_{+}}-x^{(m_{*,j})_{-}}\mid j=1,\ldots ,s)$ be the
ideal of $A$ generated by the binomials defined by the columns of $M$
and let $I(\bm)=(x^{u}-x^{v}\mid u,v\in\bn^{n}_0, u-v\in\bm )$ be the
lattice ideal of $A$ associated to $\bm$. The following hold:
\begin{itemize}
\item[$(a)$] $I(M)\subseteq I(\bm)$ and $I(\bm)=I(M):x^{\infty}$. In
  particular, $I(\bm)B\cap A=I(\bm)$;
\item[$(b)$] $I(M)B\equiv (x^{m_{*,j}}-1\mid j=1,\ldots ,s)B$
  coincides with $I(\bm)B\equiv (x^{\alpha}-1\mid \alpha\in\bm)B$;
\item[$(c)$] Given $\alpha\in\bz^n$, $\alpha\in\bm$ if and only if
  $x^{\alpha}-1\in I(M)B$;
\item[$(d)$] If $N$ is an $n\times r$ integer matrix with $I(M)=I(N)$,
  then $\bm=\mathcal{N}$.
\item[$(e)$] Let $Q$ be an $s\times s$ invertible integer matrix. If
  $MQ=T$, then $I(M)B=I(T)B$.
\end{itemize}
\end{proposition}

\begin{proof}
The containment at the beginning of $(a)$ is clear and the first
equality is \cite[Lemma~7.6]{ms}. In particular, $I(\bm)B\cap
A=I(\bm)$, because for any ideal $J$ of $A$, $JB\cap
A=J:x^{\infty}$. Since the $x_i$ are invertible in the Laurent
polynomial ring $B=S^{-1}A$, which is a flat $A$-module,
$I(\bm)B=(I(M):x^{\infty})B=I(M)B$. This proves $(b)$. If
$\alpha\in\bm$, then $x^{\alpha}-1\in I(\bm)B=I(M)B$, by item
$(b)$. Conversely, take $x^{\alpha}-1\in I(M)B=I(\bm)B$. Let
$\rho:\bm\to k^{*}$ be the trivial character and
$L_{\rho}=\bm$. Following the notation in \cite[\S~2]{es}, $I(\bm)B$
is the Laurent binomial ideal $I(\rho)$. The argument in the second
paragraph of \cite[Theorem~2.1(a), p.~13, last line]{es} shows that
$\alpha\in\bm$. This proves $(c)$. Suppose now $I(M)=I(N)$ and take
$\alpha\in\bm$. Then, by $(c)$, $x^{\alpha}-1\in I(M)B=I(N)B$.  By
$(c)$ again, this implies that $\alpha\in\mathcal{N}$, so that
$\bm\subseteq \mathcal{N}$. Analogously, $\mathcal{N}\subseteq
\bm$. This proves $(d)$. Finally, if $MQ=T$ with $Q$ invertible, then
$\bm=\bt$ and, by $(b)$, $I(M)B=I(\bm)B=I(\bt)B=I(T)B$.
\end{proof}

With this terminology, we see that Proposition~\ref{S(I)}$(c)$ says
that the hull of a PCB ideal is the lattice ideal of the lattice
spanned by the columns of the PCB matrix. That is, in concrete terms,
we have the following.

\begin{corollary}\label{S(I)isI(bl)} 
Let $I$ the PCB ideal of $A$ associated to $L$. Then $S(I)=I(\bl)$,
where $\bl\subseteq \bz^n$ is the lattice spanned by the columns of
$L$.
\end{corollary}

\begin{proof}
By Proposition~\ref{S(I)}$(c)$, with $\alpha=(1,\ldots ,1)$, and
Proposition~\ref{bm}$(a)$, $S(I)=I(L):x^{\infty}=I(\bl)$.
\end{proof}

We give now an explicit description of $S(I)$ and thus of $\hull(I)$
(see \cite[Problem~6.3]{es}).

\begin{proposition}\label{explicitsi}
Let $I=(\underline{f})=(f_1,\ldots ,f_n)$ be a PCB ideal and set
$J=(f_1,\ldots ,f_{n-1})$. Set $b(n)=(0,a_{2,2}-a_{2,3}-\ldots
-a_{2,n},\ldots ,a_{n-1,n-1}-a_{n-1,n},0)$. Then
$S(I)=I:x^{b(n)}=J:x^{b(n)}$.
\end{proposition}
\begin{proof}
By Proposition~\ref{gradeI}$(a)$, $f_1,\ldots ,f_{n-1}$ is a regular
sequence in $A$. Hence $J$ is a complete intersection and an unmixed
ideal of height $n-1$.

If $n=2$, $I$ is principal and unmixed, and $J=I=S(I)$. Moreover,
$b(n)=(0,0)$ and $J:x^{b(n)}=I:x^{b(n)}=S(I)$.

Set $n\geq 3$, so $b(n)\neq 0$. By Remark~\ref{explicitrelation},
$x^{b(n)}f_n\in J$. Hence $x^{b(n)}I\subseteq J$. By
Proposition~\ref{S(I)}$(c)$,
\begin{eqnarray*}\label{inclusions}
I\subseteq J:x^{b(n)}\subseteq I:x^{b(n)}\subseteq
I:(x^{b(n)})^{\infty}=S(I).
\end{eqnarray*}
In particular, $J:x^{b(n)}$ is a proper ideal. By the properties of
the colon operation vis-\`a-vis intersection of ideals, since $J$ is
unmixed, it follows that $J:x^{b(n)}$ is unmixed with associated
primes a (non-empty) subset of the primes associated to $J$, and hence
each of height $n-1$.

Moreover, if $\fp$ is an associated prime of $J:x^{b(n)}$, since
$I\subseteq J:x^{b(n)}$, then $I\subseteq \fp$ and, since
$\height(\fp)=n-1$, $\fp$ is a minimal prime over $I$. In particular,
$x^{b(n)}\not\in\fp$ and
$(J:x^{b(n)})_{\mathfrak{p}}=J_{\mathfrak{p}}$.

Therefore, for any associated prime $\fp$ of $J:x^{b(n)}$ (so that
$\fp$ is a minimal prime over $I$),
\begin{eqnarray*}
I_{\mathfrak{p}}\subseteq
(J:x^{b(n)})_{\mathfrak{p}}=J_{\mathfrak{p}}\subseteq
I_{\mathfrak{p}}=S(I)_{\mathfrak{p}}.
\end{eqnarray*}
Hence $(J:x^{b(n)})_{\mathfrak{p}}=S(I)_{\mathfrak{p}}$ for all
associated primes $\fp$ of $J:x^{b(n)}$, so
$J:x^{b(n)}=I:x^{b(n)}=S(I)$.
\end{proof}

The next result is a kind of ad-hoc ``unmixedness test''. For a more
general result, see the Unmixedness Test of W.V. Vasconcelos in
\cite[p.~76]{vasconcelos}.

\begin{corollary}\label{test}
Let $I$ be a PCB ideal of $A$. Then the following conditions are
equivalent.
\begin{itemize}
\item[$(i)$] $I$ is unmixed;
\item[$(ii)$] Each of $x_{1},\ldots ,x_n$ is regular modulo $I$;
\item[$(iii)$] $I=I:x_1$.
\end{itemize}
\end{corollary}
\begin{proof}
If $I$ is unmixed and if $x_i$ were a zero-divisor modulo $I$, then
$x_i$ would be in an associated prime $\fp$ of $I$ and $\fp$ would be
equal to $\fm$, a contradiction. If $I=I:x_1$, then clearly
$I=I:x_1^{\infty}$. By Proposition~\ref{S(I)}$(c)$,
$S(I)=I:x_1^{\infty}$. Thus $I=S(I)$ and, by Proposition~\ref{S(I)}$(b)$,
$I$ is unmixed.
\end{proof}

Let us state the last result in terms of lattice ideals (cf. also
\cite[Corollary~2.5]{es} or \cite[Theorem~3.2]{lv}).

\begin{corollary}\label{S(I)vslatticeideal}
Let $I$ be a PCB ideal of $A$. Then $I$ is unmixed if and only if $I$
is a lattice ideal.
\end{corollary}

\begin{proof}
If $I$ is unmixed, by Proposition~\ref{S(I)}$(b)$, $I=S(I)$ and, by
Corollary~\ref{S(I)isI(bl)}, $S(I)$ is a lattice ideal. Conversely, if
$I=I(\bm)$ is a lattice ideal, then $S(I)=IB\cap A=I(\bm)B\cap A$,
which, by Proposition~\ref{bm}, is equal to $I(\bm)=I$. Hence,
$S(I)=I$ and, by Proposition~\ref{S(I)}$(b)$, $I$ is unmixed.
\end{proof}

\begin{remark}\label{unmixedifn=3}
{\rm Let $I$ be a PCB ideal of $A$. If $n\leq 3$, $I$ is unmixed. This
  follows from Remark~\ref{casen=2} for the case $n=2$, and the fact
  that, for $n=3$, PCB ideals are ideals of Herzog-Northcott type (cf.
  \cite[Proposition~2.2(b)]{op2}). }\end{remark}

\begin{proposition}\label{mixedifn>3}
Let $I=(\underline{f})$ be a PCB ideal of $A$, $n\geq 4$. Set
$g_1=x_2^{a_{2,1}}x_3^{a_{3,1}}\cdots x_{n-1}^{a_{n-1,1}}$ and
$g_2=x_2^{a_{2,n}}x_3^{a_{3,n}}\cdots x_{n-1}^{a_{n-1,n}}$. Let
$g=x_1^{a_{1,1}-1}x_n^{a_{n,n}-a_{n,1}}-x_1^{a_{1,n}-1}g_1g_2$. Then
$g\in (I:x_1)\setminus I$. In particular, $I$ is not unmixed.
\end{proposition}
\begin{proof}
It is easy to check that
$x_1g=x_n^{a_{n,n}-a_{n,1}}f_1+g_1f_n$. Moreover, if $g\in I$, setting
$x_i=0$ for $i=2,\ldots ,n-1$, it would follow that
$x_{1}^{a_{1,1}-1}x_{n}^{a_{n,n}-a_{n,1}}$ lies in
$(x_1^{a_{1,1}},x_{n}^{a_{n,n}})k[x_1,x_n]$, a contradiction. Thus
$g\in (I:x_1)\setminus I$. By Corollary~\ref{test}, $I$ is not
unmixed. (Observe that the condition $n\geq 4$ is essential, for if
$n=3$, the ideal obtained from $I$ when substituting $x_2$ by $0$ is
$(x_1^{a_{1,1}},x_1^{a_{1,2}}x_3^{a_{3,2}},x_{3}^{a_{3,3}})k[x_1,x_3]$.)
\end{proof}

\begin{example}\label{simplestn=4}
{\rm Let $I=(x_1^3-x_2x_3x_4,x_2^3-x_1x_3x_4,x_3^3-x_1x_2x_4,
  x_4^3-x_1x_2x_3)\subset A$ be the ``simplest'' PCB ideal in
  dimension 4. By Proposition~\ref{mixedifn>3}, $I$ is not unmixed. In
  fact, the element $g\in (I:x_1)\setminus I$ built in the proof there
  is $x_1^2x_4^2-x_2^2x_3^2$. A computation with Singular
  (see \cite{singular}) shows that $I:x_1=I+(x_1^2x_2^2-x_3^2x_4^2,
  x_1^2x_3^2-x_2^2x_4^2,x_1^2x_4^2-x_2^2x_3^2)$ and that
  $I:x_1=I:x_1^2$. In particular, by Proposition~\ref{S(I)}$(c)$,
  $S(I)=I:x_1$. Alternatively, from Proposition~\ref{explicitsi}, we
  get another description of $S(I)$, namely, since $b(4)=(0,1,2,0)$,
  $S(I)=I:(x_2x_3^2)$.

On the other hand, clearly $m(I)=(16, 16,16,16)$ and so
$d=\gcd(m(I))=16$. We will see (cf. Theorem~\ref{main} below) that,
provided $k=\bc$, $I$ has exactly sixteen isolated primary components
and one irredundant embedded primary component. The next result says
that $\fc=I+(x_1)=(x_1,x_2x_3x_4,x_2^3,x_3^3,x_4^3)$ is an embedded
primary component of $I$. Alternatively, $I+(x_2x_3^2)$ is another
embedded primary component of $I$.  }\end{example}

We now give an explicit description of an irredundant embedded
component of $I$, provided that $n\geq 4$, that is independent of the
characteristic of $k$. Note that in this case, each irredundant
primary decomposition of $I$ has precisely one embedded component.

\begin{theorem}\label{embeddedcomp}
Let $I=(\underline{f})$ be a PCB ideal of $A$, $n\geq 4$. Suppose that
$I:x^{\alpha}=I:(x^{\alpha})^{\infty}$ for some
$\alpha\in\bn_0^n\setminus \{0\}$. Then the following hold.
\begin{itemize}
\item[$(a)$] $I+(x^{\alpha})$ is an irredundant $\fm$-primary
  component of $I$;
\item[$(b)$] In particular, for $b(n)=(0,a_{2,2}-a_{2,3}-\ldots
  -a_{2,n},\ldots ,a_{n-1,n-1}-a_{n-1,n},0)$, $I+(x^{b(n)})$ is an
  irredundant $\fm$-primary component of $I$.
\end{itemize}
\end{theorem}
\begin{proof}
By Proposition~\ref{S(I)}, $S(I)=\hull(I)$ and
$S(I)=I:(x^{\alpha})^{\infty}=I:x^{\alpha}$.  Moreover, since $n\geq
4$, by Proposition~\ref{mixedifn>3}, $I$ is not unmixed.

Since $I:x^{\alpha}=I:(x^{\alpha})^{\infty}$, by
\cite[Proposition~7.2$(a)$]{es}, $I=(I:x^{\alpha})\cap
(I+(x^{\alpha}))$, so $I=S(I)\cap (I+(x^{\alpha}))$, where
$S(I)=\hull(I)$ is the intersection of the isolated primary components
of $I$. Since $I$ is not unmixed, $I+(x^{\alpha})$ is not redundant.

Clearly, $\rad(I,x^{\alpha})=\fm$. Thus $I+(x^{\alpha})$ is
$\fm$-primary. One deduces that $I+(x^{\alpha})$ is an irredundant
$\fm$-primary component of $I$.

By Proposition~\ref{explicitsi}, $S(I)=I:x^{b(n)}$, i.e.,
$I:x^{b(n)}=I:(x^{b(n)})^{\infty}$. It follows, that $I+(x^{b(n)})$ is
an irredundant $\fm$-primary component of $I$.
\end{proof}

\begin{example}\label{onecomp}{\rm 
Let $I=(\underline{f})=(x_1^4-x_2x_3x_4,x_2^4-x_1^2x_3x_4,
x_3^3-x_1x_2^2x_4, x_4^3-x_1x_2x_3)\subset A$. Again, by
Proposition~\ref{mixedifn>3}, $I$ is not unmixed. Since
$b(n)=(0,1,2,0)$, by Theorem~\ref{embeddedcomp}, $I+(x_2x_3^2)$ is an
irredundant $\fm$-primary component of $I$.  On the other hand, the
integer vector associated to $I$ is $m(I)=(20,24,31,25)$ and its
greatest common divisor is $d=\gcd(m(I))=1$. By
Proposition~\ref{gradeI}, $\fp_{m(I)}=\ker(\varphi_{m(I)})$ is the
unique Herzog ideal containing $I$. Recall that the natural map
$\varphi_{m(I)}:A\to k[t]$ sends $x_1,x_2,x_3$ and $x_4$ to
$t^{20},t^{24},t^{31}$ and $t^{25}$, respectively. Therefore
$I\subseteq \fp_{m(I)}\cap \fc$. We will see
(cf. Corollary~\ref{S(I)prime} below) that, since $d=1$,
$S(I)=\fp_{m(I)}$, so $I=S(I)\cap \fc=\fp_{m(I)}\cap \fc$, an irredundant
intersection, and the previous inclusion is an equality.
}\end{example}

\section{A review of the normal decomposition of an 
integer matrix}\label{ashort}

In this section we review some well-known facts about linear algebra
over $\bz$ or, more generally, over a Principal Ideal Domain. Our
general reference is \cite[Chapter~3]{jacobson}. As before,
$A=k[\underline{x}]$ is the polynomial ring in $n$ variables
$\underline{x}=x_1,\ldots ,x_n$ over a field $k$, $n\geq 2$.

\begin{definition}{\rm 
Let $M$ be a non-zero $n\times s$ integer matrix. Then there exists an
$n\times n$ invertible integer matrix $P$ and an $s\times s$
invertible integer matrix $Q$ such that $PMQ=D$, where $D$ is an
$n\times s$ integer diagonal matrix $D=\diag(d_1,d_2,\ldots ,
d_r,0,\ldots ,0)$, where $d_i\in\bn$ and $d_i\mid d_j$ if $i\leq j$,
and $r=\rank(M)$. The matrix $D$ is called a {\em normal form} of $M$
and the expression $PMQ=D$ a {\em normal decomposition} of $M$ ($D$ is
also called the Smith Normal Form of $M$, see, e.g., \cite{singular}).
}\end{definition}

\begin{remark}\label{invariant}{\rm 
The non-zero diagonal elements of a normal form $D$ of $M$, referred
to as the {\em invariant factors} of $M$, are unique. Indeed, let
$I_{t}(M)$ be the ideal of $\bz$ generated by the $t\times t$-minors
of the matrix $M$, $I_{t}(M):=\bz$ for $t\leq 0$ and $I_{t}(M):=0$ for
$t> \min(n,s)$. Then $I_{t}(M)=I_{t}(PMQ)$ for all invertible integer
matrices $P$ and $Q$ (see, e.g., \cite[Chapter~5, Lemma~4.8 and
  Exercise~10, pp.~232-233]{clo}). In particular, $I_t(M)=I_t(D)$ and
$\gcd(I_{t}(M))=\gcd(I_t(D))$, understanding by the $\gcd(J)$ of a
non-zero ideal $J$ of $\bz$ its non-negative generator (and setting
$0$ to be the $\gcd$ of the zero ideal). Therefore, setting
$\Delta_t=\gcd(I_t(M))$, one has that $d_1=\Delta_1$,
$d_2=\Delta_2\Delta_1^{-1},\ldots ,d_r=\Delta_{r}\Delta_{r-1}^{-1}$,
where $r=\rank(M)$ (see, e.g., \cite[Theorems~3.8 and
  3.9]{jacobson}). Observe that, in particular, $d_1=\Delta_1$,
$d_1d_2=\Delta_2,\ldots ,d_1\cdots d_{r}=\Delta_r$.  }\end{remark}

\begin{lemma}\label{lastrow}
Let $I$ be the PCB ideal of $A$ associated to $L$. Let $m(I)$ be the
integer vector associated to $I$, $d=\gcd(m(I))$ and
$\nu(I)=m(I)/d$. Let $PLQ=D$ be a normal decomposition of $L$ and
$d_1,\ldots ,d_{n-1}$ the invariant factors of $L$. Then $d=d_1\cdots
d_{n-1}$. Moreover, the last row of $P$ is $\pm \nu(I)$.
\end{lemma}
\begin{proof}
Observe that, by Lemma~\ref{rankofL}, $\rank(L)=n-1$ and hence there
are $n-1$ (non-zero) invariant factors $d_1,\ldots ,d_{n-1}$. By
Remark~\ref{invariant}, $d_1\cdots
d_{n-1}=\Delta_{n-1}=\gcd(I_{n-1}(L))$. But the $(n-1)\times (n-1)$
minors of $L$ are precisely the entries of the matrix $\adj(L)$, each
of whose rows is equal to the last one, denoted by $m(I)$ (see
Lemma~\ref{rankofL} and Definition~\ref{defm(I)}). Thus $d_1\cdots
d_{n-1}=\Delta_{n-1}=\gcd(I_{n-1}(L))=\gcd(m(I))=d$.

Since $PL=DQ^{-1}$ and the last row of $DQ^{-1}$ is zero, $p_{n,*}L=0$
and $p_{n,*}^{\top}\in\nulls(L^{\top})$. By Lemma~\ref{rankofL} and
Definition~\ref{defm(I)}, $\nulls(L^{\top})$ is generated, as a
$\bq$-linear subspace, by the vector $m(I)^{\top}$, or equivalently,
by the vector $\nu(I)^{\top}=m(I)^{\top}/d$. Therefore there exist
$r,s\in\bz\setminus\{0\}$ such that $rp_{n,*}=s\nu(I)$. Observe that,
since $\det(P)=\pm 1$, then $\gcd(p_{n,*})=1$. Now, taking the
greatest common divisor, we get $r=\pm s$ and hence $p_{n,*}=\pm
\nu(I)$.
\end{proof}

\begin{example}\label{Pn=2}
{\rm Let $I$ be the PCB ideal of $A$ associated to $L$. Suppose that
  $n=2$. Then $a_{1,1}=a_{1,2}$ and $a_{2,1}=a_{2,2}$. Moreover
  $m(I)=(a_{2,2,},a_{1,1})$. Let $d=\gcd(m(I))$ and write
  $a_{i,i}=da_{i,i}^{\prime}$ and $d=b_1a_{1,1}+b_2a_{2,2}$ for some
  $b_1,b_2\in\bz$. The invariant factor of $L$ is $d_1=d$ and
\begin{eqnarray*}
\left( \begin{array}{rr}b_1&-b_2\\a_{2,2}^{\prime}&a_{1,1}^{\prime}
\end{array}\right)
\left( \begin{array}{rr}a_{1,1}&-a_{1,2}\\-a_{2,1}&a_{2,2}\end{array}\right)
\left( \begin{array}{rr}1&1\\0&1\end{array}\right)=
  \left( \begin{array}{rr}d&0\\0&0\end{array}\right)
\end{eqnarray*}
is a normal decomposition of $L$.
}\end{example}

In order to describe the isolated components of the PCB ideal $I$
associated to $L$, it will be convenient to know the entries of a
matrix $P$ in a normal decomposition $PLQ=D$ of $L$ (see
Theorem~\ref{main}).

If $n=3$ and if the entries of a row of $L$ are relatively prime, or
more generally if their greatest common divisor equals the first
invariant factor $d_1$, we see next that obtaining $P$ explicitly is
almost immediate.  Observe that the example also covers the situation
where $\gcd(a_{2,1},a_{2,3})=d_1$ or $\gcd(a_{1,2},a_{1,3})=d_1$, via
an appropriate relabelling of the suffices. However, calculating an
explicit normal decomposition of a general matrix $L$, even for $n=3$,
is technical and unilluminating. For concrete instances of the matrix
$L$, a normal decomposition of $L$ can be obtained for example in
Singular (see `smithNormalForm' \cite{singular}).

\begin{example}\label{Pn=3}
{\rm Let $I$ be the PCB ideal of $A$ associated to $L$. Suppose that
  $n=3$.  Let $m(I)$ be the integer vector associated to $I$,
  $d=\gcd(m(I))$ and $\nu(I)=m(I)/d$. Let $d_1,d_{2}$ be the invariant
  factors of $L$. In particular, $d_1=\gcd(I_1(L))$ and $d_1d_2=d$. Set
  $d_2^{\prime}=d_2/d_1$. Let
  $b=\gcd(a_{3,1},a_{3,2})=b^{\prime}d_1$. Let $c_1,c_2\in\bz$ with
  $b=c_1a_{3,1}+c_2a_{3,2}$. Set
  $\alpha_1=-c_1a_{1,1}+c_2a_{1,2}=d_1\alpha_1^{\prime}$ and
  $\alpha_2=c_1a_{2,1}-c_2a_{2,2}=d_1\alpha_2^{\prime}$, for some
  $\alpha_1^{\prime},\alpha_2^{\prime}\in\bz$. The following
  conditions are equivalent:
\begin{itemize}
\item[$(a)$] $\gcd(a_{3,1},a_{3,2})=d_1$;
\item[$(b)$] $\gcd(\nu_1,\nu_2)=1$, $b^{\prime}\mid c$ and
  $(b^{\prime})^{2}\mid d_2^{\prime}$, where
  $c=s_2\alpha_1^{\prime}-s_1\alpha_2^{\prime}$ and $s_1,s_2\in\bz$
  are such that $s_1\nu_1+s_2\nu_2=1$;
\item[$(c)$] There exists a normal decomposition $PLQ=D$ of $L$ with
  the first row of $P$ equal to $(0,0,1)$.
\end{itemize}
Moreover, in this particular case, the second row of $P$ is given by
$(s_2,-s_1,-c)$, while the third row is given by
$\nu(I)=(\nu_1,\nu_2,\nu_3)$.  }\end{example}
\begin{proof}
Set $a_{3,1}=b\tilde{a}_{3,1}$, $a_{3,2}=b\tilde{a}_{3,2}$, with
$\tilde{a}_{3,1},\tilde{a}_{3,2}\in\bn$. Let
$Q=\left(\begin{array}{rrr}-c_1&\tilde{a}_{3,2}&1\\-c_2&-\tilde{a}_{3,1}&1
  \\0&0&1\end{array}\right)$. Then $\det(Q)=1$ and
\begin{eqnarray*}
LQ=\left(\begin{array}{rrr}a_{1,1}&-a_{1,2}&-a_{1,3}
  \\-a_{2,1}&a_{2,2}&-a_{2,3}\\-a_{3,1}&-a_{3,2}&a_{3,3}\end{array}\right)
\left(\begin{array}{rrr}-c_1&\tilde{a}_{3,2}&1\\-c_2&-\tilde{a}_{3,1}&1
  \\0&0&1
\end{array}\right)=
  \left(\begin{array}{rrc}\alpha_1&m_2/b&0
    \\\alpha_2&-m_1/b&0\\b&0&0\end{array}\right).
\end{eqnarray*}
Since $\nu(I)^{\top}$ is a $\bq$-basis of $\nulls(L^{\top})$
(cf. Definition~\ref{defm(I)}), $\nu(I)LQ=0$ and
\begin{eqnarray}\label{equality}
\alpha_1\nu_1+\alpha_2\nu_2+b\nu_3=0.
\end{eqnarray}

For $P_1=\left(\begin{array}{rrr}0&0&1\\0&1&0\\1&0&0
\end{array}\right)$, $\det(P_1)=-1$ and $P_1LQ=
  \left(\begin{array}{lrc}b&0&0\\ \alpha_2&-m_1/b&0\\ \alpha_1&m_2/b&0
\end{array}\right)$.

For $P_2=\left(\begin{array}{rrr}1&0&0\\-\alpha_2^{\prime}&1&0
  \\-\alpha_1^{\prime}&0&1
\end{array}\right)$, $\det(P_2)=1$ and $P_2P_1LQ=
\left(\begin{array}{crc}b&0&0\\\alpha_2^{\prime}(d_1-b)&-m_1/b&0
  \\ \alpha_1^{\prime}(d_1-b)&m_2/b&0\end{array}\right)$.

The unique non-zero $2\times 2$ minor of $P_2P_1LQ$ defined by the
last two rows is, up to sign, equal to
$((d_1-b)d_2/b)[\alpha_1\nu_1+\alpha_2\nu_2]$, which, by the equality
(\ref{equality}) above, is equal to $-(d_1-b)d_2\nu_3$.  Since
$I_2(D)=I_2(L)=I_2(P_2P_1LQ)$ (cf. Remark~\ref{invariant}),
\begin{eqnarray*}
d=d_1d_2=\gcd(I_2(D))=\gcd(I_2(L))=\gcd(I_2(P_2P_1LQ))=
\gcd(\nu_1d,\nu_2d,(b-d_1)d_2\nu_3).
\end{eqnarray*} 
Since $b=d_1b^{\prime}$, then $d=d\cdot
\gcd(\nu_1,\nu_2,(b^{\prime}-1)\nu_3)$.  Therefore,
$\gcd(\nu_1,\nu_2,(b^{\prime}-1)\nu_3)=1$ and
$\gcd(\nu_1,\nu_2,(b^{\prime}-1))=1$.

Observe that until now we have not used any of the hypotheses $(a)$,
$(b)$ or $(c)$. Now suppose that $\gcd(a_{3,1},a_{3,2})=d_1$. Then
$b^{\prime}=1$ and $\gcd(\nu_1,\nu_2)=1$. Thus $(a)$ implies $(b)$.

Suppose now that $\gcd(\nu_1,\nu_2)=1$ (where $b$ is not assumed a
priori to be equal to $d_1$). Let $s_1,s_2\in\bz$ with
$s_1\nu_1+s_2\nu_2=1$. Set
$c=s_2\alpha_1^{\prime}-s_1\alpha_2^{\prime}$.  Let
$P_3=\left(\begin{array}{crr}
  1&0&0\\c&-s_1&s_2\\(1-b^{\prime})\nu_3&\nu_2&\nu_1
\end{array}\right)$. Then $\det(P_3)=-1$ and $P_3P_2P_1LQ=\left(
\begin{array}{ccc}d_1b^{\prime}&0&0\\
d_1c&d_2/b^{\prime}&0\\0&0&0\end{array}\right)$, using the
equality~(\ref{equality}) above.

Suppose that $b^{\prime}\mid c$. Set $\tilde{c}=c/b^{\prime}$ and
$P_4=\left(\begin{array}{rrr}1&0&0\\-\tilde{c}&1&0\\0&0&1\end{array}
  \right)$. Then $\det(P_4)=1$. Set $P=P_4P_3P_2P_1$. Then
  $P=\left(\begin{array}{crr}
    0&0&1\\s_2&-s_1&-\tilde{c}\\\nu_1&\nu_2&\nu_3
\end{array}\right)$ and
$PLQ=\left(
\begin{array}{ccc}d_1b^{\prime}&0&0\\0&d_2/b^{\prime}&0\\0&0&0
\end{array}\right)$.
If $(b^{\prime})^{2}\mid d_2^{\prime}$, then $d_1b^{\prime}$ and
$d_2/b^{\prime}$ are positive integers with $d_1b^{\prime}\mid
(d_{2}/b^{\prime})$. By the unicity of the normal form of $L$,
$b^{\prime}=1$. Therefore $PLQ=D$ is a normal decomposition of $L$ and
$(b)$ implies $(c)$.

Finally, suppose that there exists a normal decomposition $PLQ=D$ of
$L$, with the first row of $P$ equal to $(0,0,1)$. Equating the first
rows of the identity $PL=DQ^{-1}$, one has that, if
$Q^{-1}=(u_{i,j})$,
$(-a_{3,1},-a_{3,2},a_{3,3})=(d_1u_{1,1},d_1u_{1,2},d_1u_{1,3})$. Therefore
\begin{eqnarray*}
\gcd(a_{3,1},a_{3,2},a_{3,3})=d_1\cdot \gcd(u_{1,1},u_{1,2},u_{1,3})=d_1
\end{eqnarray*}
 and $(c)$ implies $(a)$.
\end{proof}

We finish the section with the answer to a question posed by Josep
\`Alvarez Montaner. Denote by $\fit_i(\bm)$ the $i$-th Fitting ideal
of a $\bz$-module $\bm$ (see, e.g., \cite[Definition~5.4.9]{clo}).

\begin{proposition}\label{torsion}
Let $I$ be the PCB ideal of $A$ associated to $L$. Let $m(I)$ be the
integer vector associated to $I$, $d=\gcd(m(I))$.  Let $\bl$ be the
lattice of $\bz^n$ spanned by the columns of $L$. Then the following
hold.
\begin{itemize}
\item[$(a)$] $\fit_1(\bz^n/\bl)=d\bz$ and $\fit_0(\bz^n/\bl)=0$.
\item[$(b)$] $\bz^{n}/\bl\cong \bz\oplus \bz/d_1\bz\oplus \ldots
  \oplus \bz/d_{n-1}\bz$, with $d_1,\ldots ,d_{n-1}$ the invariant
  factors of $L$.
\item[$(c)$] The cardinality of the torsion group of $\bz^n/\bl$ is
  $d$.
\item[$(d)$] $\bl$ is a direct summand of $\bz^{n}$ if and only if
  $d=1$.
\end{itemize}
\end{proposition}
\begin{proof} 
Let $PLQ=D$ be a normal decomposition of $L$ and $d_1,\ldots ,d_{n-1}$
the invariant factors of $L$. By Lemma~\ref{lastrow}, $d=d_1\cdots
d_{n-1}$. By definition,
$\fit_{1}(\bz^n/\bl)=I_{n-1}(L)=I_{n-1}(D)=(d_1\cdots
d_{n-1})\bz=d\bz$ and $\fit_0(\bz^n/\bl)=I_{n}(L)=I_n(D)=0$. Since
$PLQ=D$ is a normal decomposition of $\bl$, the $\bz$-module
$\bz^{n}/\bl$ admits a decomposition $\bz\oplus\bt$, where
$\bt=\bz/d_1\bz\oplus \ldots \oplus\bz/d_{n-1}\bz$ is the torsion
module (see, e.g., \cite[Chapter~3]{jacobson}). Clearly $d=d_1\cdots
d_{n-1}$ is the cardinality of the torsion group of
$\bz^n/\bl$. Finally, since $\rank(L)=n-1$, then $\rank(\bl)=n-1$ and
$\rank(\bz^n/\bl)=1$ (see, e.g., \cite[\S~1.4]{bh}). Hence $\bl$ is a
direct summand of $\bz^n$ if and only if $\bz^{n}/\bl$ is a free
$\bz$-module of rank $1$. By \cite[Proposition~20.8]{eisenbud}, the
latter holds if and only if $\fit_1(\bz^n/\bl)=\bz$ and
$\fit_0(\bz^n/\bl)=0$, i.e., if and only if $d=1$.
\end{proof}

\begin{remark}{\rm 
Note that an obvious analogue of Proposition~\ref{torsion} holds for
any $n\times n$ matrix $M$ of rank $n-1$, with invariant factors
$d_1,\ldots ,d_{n-1}$, with $d$ now defined merely as the product
$d_1\cdots d_{n-1}$. Note also that the transpose $M^{\top}$ of $M$
again has $d_1,\ldots ,d_{n-1}$ as its invariant factors: indeed, any
integer matrix and its transpose have the same invariants. 
}\end{remark}

\begin{remark}\label{lvvsop}
{\rm There is an overlap between the results of Section~\ref{ashort}
  and the results of \cite[Section~3]{lv}. (Recall that an integer
  matrix and its transpose have the same invariant factors.) The
  latter results were obtained using Gr\"obner Basis Theory rather
  than the theory of Fitting ideals, and with different objectives in
  mind. For example, one can contrast the statement of
  \cite[Corollary~3.19]{lv} with the situation that obtains for the
  ideal $I$ considered in Example~\ref{onecomp} above. In
  Example~\ref{onecomp}, $d=1$ and $I(\bl)=S(I)=\fp_{m(I)}$, which is
  the kernel of the natural map $A\to k[t]$ sending $x_1,x_2,x_3$ and
  $x_4$ to $t^{20},t^{24},t^{31}$ and $t^{25}$, respectively, whereas
  in L\'opez and Villarreal's theory, $d=1$ and
  $I(\bl)=(x_1-x_2,x_1-x_3,x_1-x_4)$. Observe that the ideals in
  \cite{lv} have to be homogeneous in the standard grading, hence this
  simpler form. See also Remark~\ref{law}.}\end{remark}

\section{Applying the Eisenbud-Sturmfels theory of Laurent
 binomial ideals}\label{applying}

In this section we apply the theory of Laurent binomial ideals
developed in \cite{es}. Recall that for an $n\times s$ integer matrix
$M$, we denote by $m_{i,*}$ and $m_{*,j}$ its $i$-th row and $j$-th
column, respectively. The abelian group generated by the columns of
$M$ is denoted ${\bm}=\bz m_{*,1}+\ldots +\bz
m_{*,s}=\varphi(\bz^s)\subseteq \bz^n$, where
$\varphi:\bz^{s}\to\bz^{n}$ is the homomorphism defined by the matrix
$M$. By an $n\times n$ invertible integer matrix, we understand an
$n\times n$ matrix $P$ with entries in $\bz$ whose determinant is $\pm
1$. Thus its inverse matrix $P^{-1}$ is also an integer matrix. Set
$A=k[\underline{x}]$ to be the polynomial ring in $n$ variables
$\underline{x}=x_1,\ldots ,x_n$ over a field $k$, $n\geq 2$, and
$B=k[\underline{x}^{\pm}]=k[x_1,\ldots ,x_n,x_1^{-1},\ldots
  ,x_n^{-1}]$, the Laurent polynomial ring.

\begin{remark}\label{IBisci}{\rm 
Let $I=(\underline{f})=(f_1,\ldots ,f_n)$ be the PCB ideal of $A$
associated to $L$. Then
\begin{eqnarray*}
IB=(f_1,\ldots ,f_{n-1})B=(x^{l_{*,1}}-1,\ldots ,x^{l_{*,n-1}}-1)B.
\end{eqnarray*}
In particular, $IB$ is a complete intersection.  }\end{remark}
\begin{proof}
Clearly $IB=(f_1,\ldots ,f_n)B=(x^{l_{*,1}}-1,\ldots
,x^{l_{*,n}}-1)B$. Since the sum of the entries of each row is zero,
the last column of $L$, $l_{*,n}$, is a $\bz$-linear combination of
the first $n-1$ columns of $L$. Thus, by Remark~\ref{bm}$(c)$,
$x^{l_{*,n}}-1\in (x^{l_{*,1}}-1,\ldots ,x^{l_{*,n-1}}-1)B$. An
alternative proof would follow from Remark~\ref{explicitrelation} (see
the proof of Proposition~\ref{aci}$(b)$).
\end{proof}

We now make explicit the change of variables we will use.

\begin{lemma}\label{algindep}
Let $P=(p_{i,j})$ be an $n\times n$ invertible integer matrix. Set
$R=(r_{i,j})$, its inverse. Let $y_1=x^{r_{*,1}},\ldots
,y_n=x^{r_{*,n}}$ in $B=k[\underline{x}^{\pm}]=k[x_1,\ldots
  ,x_n,x_1^{-1},\ldots ,x_n^{-1}]$. Then
\begin{itemize}
\item[$(a)$] $x_1=y^{p_{*,1}},\ldots ,x_n=y^{p_{*,n}}$;
\item[$(b)$] $B=k[x_1,\ldots ,x_n,x_1^{-1},\ldots
  ,x_n^{-1}]=k[y_1,\ldots ,y_n,y_1^{-1},\ldots ,y_n^{-1}]$;
\item[$(c)$] $y_1,\ldots ,y_n$ are algebraically independent over $k$.
\end{itemize}
\end{lemma}
\begin{proof}
Since $RP$ is the identity matrix, $y^{p_{*,i}}=y_{1}^{p_{1,i}}\cdots
y_n^{p_{n,i}}=x^{r_{*,1}p_{1,i}}\cdots
x^{r_{*,n}p_{n,i}}=x^{Rp_{*,i}}=x_i$. Clearly
$k[\underline{y}^{\pm}]\subseteq B$ and the equality follows by part
$(a)$. Writing $Q(R)$ to denote the quotient field of a domain $R$, we
have $Q(A)=Q(B)=Q(k[\underline{y}^{\pm}])=Q(k[\underline{y}])$. Thus
$\dim A=\trdeg_{k}(Q(A))=\trdeg_k(Q(k[\underline{y}]))$ and the
transcendence degree of $k[y_1,\ldots ,y_n]$ over $k$ is $n$. It
follows, e.g., using Noether's Normalization Lemma, that $y_1,\ldots
,y_n$ are algebraically independent over $k$.
\end{proof}

The next result expresses $IB$ in terms of the new variables.

\begin{lemma}\label{easyexpofIB}
Let $I$ be the PCB ideal of $A$ associated to $L$. Let $PLQ=D$ be a
normal decomposition of $L$ and $d_1,\ldots ,d_{n-1}$ the invariant
factors of $L$. Let $R=(r_{i,j})$ be the inverse of $P$. Set
$y_1=x^{r_{*,1}},\ldots ,y_n=x^{r_{*,n}}$ in
$B=k[\underline{x}^{\pm}]$. Then $IB=(y_1^{d_1}-1,\ldots
,y_{n-1}^{d_{n-1}}-1)B$.
\end{lemma}
\begin{proof}
By Remark~\ref{IBisci}, $IB=(x^{l_{*,1}}-1,\ldots
,x^{l_{*,n-1}}-1)B$. Using Lemma~\ref{algindep}$(a)$ and substituting
$x_i$ by $y^{p_{*,i}}$, we get $x^{l_{*,i}}=x_1^{-a_{1,i}}\cdots
x_i^{a_{i,i}}\cdots x_n^{-a_{n,i}}= y^{p_{*,1}(-a_{1,i})}\cdots
y^{p_{*,i}a_{i,i}}\cdots
y^{p_{*,n}(-a_{n,i})}=y^{Pl_{*,i}}$. Therefore the ideal
$(x^{l_{*,1}}-1,\ldots ,x^{l_{*,n-1}}-1)B$ is equal to
$(y^{Pl_{*,1}}-1,\ldots ,y^{Pl_{*,n-1}}-1)B$, and so equal to
$(y^{(DQ^{-1})_{*,1}}-1,\ldots ,y^{(DQ^{-1})_{*,n-1}}-1)B$. By
Remark~\ref{bm}$(e)$, the latter is equal to the ideal
$(y^{D_{*,1}}-1,\ldots ,y^{D_{*,n-1}}-1)B= (y_1^{d_{1}}-1,\ldots
,y_{n-1}^{d_{n-1}}-1)B$.
\end{proof}

Our aim now is to give a minimal primary decomposition of $IB$ in
terms of these new variables. Before this, we introduce some notation.

\begin{notation}\label{blambda}
{\rm Let $I$ be the PCB ideal associated to $L$, $m(I)$ its associated
  integer vector, $d=\gcd(m(I))$ and $\nu(I)=m(I)/d$. Let $PLQ=D$ be a
  normal decomposition of $L$ with $p_{n,*}=\nu(I)$ (see
  Lemma~\ref{lastrow}). Let $d_1,\ldots ,d_{n-1}$ be the invariant
  factors of $L$. Let $R=(r_{i,j})$ be the inverse of $P$ and set
  $y_1=x^{r_{*,1}},\ldots ,y_n=x^{r_{*,n}}$ in
  $B=k[\underline{x}^{\pm}]$. For any $\lambda=(\lambda_1,\ldots
  ,\lambda_{n-1})\in (k^{*})^{n-1}$, $k^{*}:=k\setminus\{0\}$, set
  $\fb_{\lambda}=(y_1-\lambda_1,\ldots
  ,y_{n-1}-\lambda_{n-1})B$. Clearly $\fb_{\lambda}$ is a prime ideal
  of $B$ of height $n-1$. In particular, $\fb_{\lambda}^{c}$ is a
  prime ideal of $A$ of height $n-1$.

Suppose that $k$ contains the $d_{n-1}$-th roots of unity. (Note that
then $k$ will also contain the $d_1$-th$,\ldots ,d_{n-2}$-th roots of
unity, respectively, since $d_{i}\mid d_{i+1}$ for $i=1,\ldots ,n-2$.)
We will write $\{\xi_{i,1},\ldots ,\xi_{i,d_{i}}\}$ to denote the set
of $d_i$-th roots of unity in $k$ when these exist and are distinct,
and set $\Lambda(D)=\prod_{i=1}^{n-1}\{\xi_{i,1},\ldots
,\xi_{i,d_{i}}\}\subset (k^{*})^{n-1}$. Clearly, if the characteristic
of $k$, $\ch(k)$, is zero or $\ch(k)=p$, $p$ a prime with $p\nmid
d_{n-1}$, then the cardinality of $\Lambda(D)$ is $d_1\cdots d_{n-1}$,
which, by Lemma~\ref{lastrow}, is equal to $d$.  }\end{notation}

\begin{theorem}\label{dcompofSI}
Let $I$ be the PCB ideal associated to $L$, $m(I)$ its associated
integer vector, $d=\gcd(m(I))$ and $\nu(I)=m(I)/d$. Let $d_1,\ldots
,d_{n-1}$ be the invariant factors of $L$. With the notations above:
\begin{itemize}
\item[$(a)$] Suppose that $k$ contains the $d_{n-1}$-th roots of unity
  and that the characteristic of $k$, $\ch(k)$, is zero or $\ch(k)=p$,
  $p$ a prime with $p\nmid d_{n-1}$. Then $IB=\cap
  _{\lambda\in\Lambda(D)}\fb_{\lambda}$ and $S(I)=\cap
  _{\lambda\in\Lambda(D)}\fb^{c}_{\lambda}$ are minimal primary
  decompositions. In particular, $IB$ and $S(I)$ are unmixed, radical
  and have exactly $d$ distinct associated primes.
\item[$(b)$] If $k$ is an arbitrary field, then $IB$ and $S(I)$ have
  at most $d$ distinct associated primes.
\end{itemize}
\end{theorem}
\begin{proof}
As in Notation~\ref{blambda}, set $y_1=x^{r_{*,1}},\ldots
,y_n=x^{r_{*,n}}$ in $B=k[\underline{x}^{\pm}]$.  Lemma~\ref{algindep}
says that $\underline{y}=y_1,\ldots ,y_n$ are algebraically
independent over $k$ and
$B=k[\underline{x}^{\pm}]=k[\underline{y}^{\pm}]$. By
Remark~\ref{IBisci}, $IB$ is a complete intersection and, by
Lemma~\ref{easyexpofIB}, $IB=(y_1^{d_1}-1,\ldots
,y_{n-1}^{d_{n-1}}-1)B$.

Let us prove $(a)$. Consider
$\Lambda(D)=\prod_{i=1}^{n-1}\{\xi_{i,1},\ldots
,\xi_{i,d_{i}}\}\subset (k^{*})^{n-1}$ and, for any $\lambda\in
\Lambda(D)$, $\fb_{\lambda}=(y_1-\lambda_1,\ldots
,y_{n-1}-\lambda_{n-1})B$, as in Notation~\ref{blambda}. If $\fp$ is
any prime ideal containing $IB=(y_1^{d_1}-1,\ldots
,y_{n-1}^{d_{n-1}}-1)B$ then an immediate argument shows that
$\fb_{\lambda}\subseteq \fp$ for some $\lambda\in\Lambda(D)$.  Hence
the $\fb_{\lambda}$, for $\lambda\in\Lambda(D)$, are the minimal
primes containing $IB$. Since $IB$ is a complete intersection, the
$\fb_{\lambda}$ are also the associated primes of $IB$.

Consider the inclusion $IB\subseteq
\cap_{\lambda\in\Lambda(D)}\fb_{\lambda}$ and localize at a particular
associated prime $\fb_{\mu}=(y_1-\mu_{1},\ldots,y_{n-1}-\mu_{n-1})B$,
with $\mu\in\Lambda(D)$, say. We see that this inclusion becomes an
equality, since, if for any $\lambda\in\Lambda(D)$ and $i\in\{1,\ldots
,n-1\}$ we have $\lambda_i\neq \mu_i$, then $y_i-\lambda_i$ becomes a
unit in the localization and, as a result, each $y_i^{d_i}-1$ becomes
an associate of $y_i-\mu_{i}$. Therefore, by
Proposition~\ref{S(I)}$(d)$, $IB=\cap
_{\lambda\in\Lambda(D)}\fb_{\lambda}$ and $S(I)=\cap
_{\lambda\in\Lambda(D)}\fb^{c}_{\lambda}$ are minimal primary
decompositions of $IB$ and $S(I)$, respectively. Recall that the
cardinality of $\Lambda(D)$ is $d_1\cdots d_{n-1}=d$ by
Lemma~\ref{lastrow}.

Now suppose that $k$ is an arbitrary field; let $\overline{k}$ denote
the algebraic closure of $k$. Note that, by Base Change, the extension
$k[\underline{y}^{\pm}]\hookrightarrow
\overline{k}[\underline{y}^{\pm}]$ is faithfully flat and integral. In
particular, by the former property,
$J:=I\overline{k}[\underline{y}^{\pm}]$ is again a complete
intersection.

There are at most $d_i$ distinct $d_i$-th roots of unity in
$\overline{k}$. The above argument yields that there are then at most
$d$ distinct minimal and so associated primes of the ideal $J$. By
standard properties of integral extensions, the number of minimal
primes containing $IB$ is at most the number of minimal primes
containing $J$, and the result follows.
\end{proof}

As an immediate corollary, we have the following analogue of
\cite[Theorem~7.8]{op2}. Note that the next result can be interpreted
as giving equivalent conditions for $S(I)$ (which is a lattice ideal
by Corollary~\ref{S(I)isI(bl)}) to be a toric ideal (in the terminology of
\cite[Chapter~7]{ms}).

\begin{corollary}\label{S(I)prime}
Let $I$ be the PCB ideal associated to $L$, $m(I)$ its associated
integer vector and $d=\gcd(m(I))$. The following conditions are
equivalent:
\begin{itemize}
\item[$(i)$] $S(I)$ is prime;
\item[$(ii)$] $S(I)=\fp_{m(I)}$;
\item[$(iii)$] $d=1$.
\end{itemize}
\end{corollary}
\begin{proof}
On one hand, by Proposition~\ref{gradeI}, $\fp_{m(I)}$ is a minimal
prime of $I$. Moreover, by Proposition~\ref{S(I)}$(a)$,
$S(I)=\hull(I)$, the intersection of the isolated primary components
of $I$. Therefore $I\subseteq S(I)\subseteq\fp_{m(I)}$ and
$\fp_{m(I)}$ is a minimal prime of $S(I)$ too. Therefore, $S(I)$ is
prime if and only if $S(I)=\fp_{m(I)}$.

Suppose that $d=1$. Then $d_{n-1}=1$ and $k$ fulfills the hypotheses
of $(a)$ in Theorem~\ref{dcompofSI}. Thus $S(I)$ is prime.

Conversely, suppose that $S(I)$ is prime and $d>1$. We will derive a
contradiction.

Now $IB$ equals the localization $S(I)_{x},$ so $IB$ is also
prime. Recall Lemmas~\ref{algindep} and \ref{easyexpofIB}, and set
$A^{\prime}=k[y_{1},\ldots, y_{n}]$.  Note that $d_{n-1}>1$ and that
$B=A_{y}^{\prime }$, where $y=y_{1}\cdots y_{n}$. Hence
$y_{n-1}^{d_{n-1}}-1\in IB\cap A^{\prime }= (y_{1}^{d_{1}}-1,\ldots
,y_{n-1}^{d_{n-1}}-1)A^{\prime}:_{A^{\prime}}y^{\infty}$. This last
ideal is a prime ideal in $A^{\prime }$ since $IB$ is prime in
$B$. Now for some $t\in\bn_{0}$, either $y^{t}(y_{n-1}-1)$ or
$y^{t}(y_{n-1}^{d_{n-1}-1}+\ldots +1)$ lies in
$(y_{1}^{d_{1}}-1,\ldots ,y_{n-1}^{d_{n-1}}-1)A^{\prime }$.

On setting each of $y_{1},\ldots ,y_{n-2}$ equal to $1$, we deduce
that either $y_{n-1}^{t}(y_{n-1}-1)$ or
$y_{n-1}^{t}(y_{n-1}^{d_{n-1}-1}+\ldots +1)$ lies in
$(y_{n-1}^{d_{n-1}}-1)k[y_{n-1}]$. Since $k[y_{n-1}]$ is a UFD, it
follows in fact that either $y_{n-1}-1\in
(y_{n-1}^{d_{n-1}}-1)k[y_{n-1}]$ or else $y_{n-1}^{d_{n-1}-1}+\ldots
+1\in (y_{n-1}^{d_{n-1}}-1)k[y_{n-1}]$. Since $d_{n-1}>1$, this is
impossible.
\end{proof}

Before proceeding to study the explicit description of
$\fb_{\lambda}^{c}$, we take advantage of this result to examine how
the class of PCB ideals plays off against the class of binomial ideals
considered by Ho\c{s}ten and Shapiro \cite{hs}. We find that the
overlap in the two classes is a trivial one (recall the notations at
the end of Section~\ref{introduction}; see also Proposition~\ref{bm}).

\begin{proposition}\label{PCBvshs}
Let $A=k[\underline{x}]$ be the polynomial ring in $n$ variables
$\underline{x}=x_1,\ldots ,x_n$. Let $\mathfrak{A}$ be the class of
binomial ideals of $A$ defined by the columns of $n\times r$ integer
matrices $M$ of rank $r$, $r\leq n$, with $\bm\cap \bn^n_0=\{ 0\}$ and
such that, as in \cite{hs}, the lattice ideal $I(\bm)$ is prime. Let
$\mathfrak{B}$ be the class of PCB ideals of $A$.
\begin{itemize}
\item[$(a)$] If $n>2$, the intersection of $\mathfrak{A}$ and
  $\mathfrak{B}$ is empty.
\item[$(b)$] If $n=2$, the intersection of $\mathfrak{A}$ and
  $\mathfrak{B}$ is the class of principal ideals generated by an
  irreducible binomial.
\end{itemize}
\end{proposition}

\begin{proof} 
Suppose that $J$ is an ideal in the intersection of $\mathfrak{A}$ and
$\mathfrak{B}$. Then $J=I(M)$, for some $n\times r$ matrix of rank
$r$, $r\leq n$, with $\bm\cap \bn^n_0=\{ 0\}$ and with $I(\bm)$ prime,
and $J=I(L)$ for an $n\times n$ PCB matrix $L$. By
Proposition~\ref{bm}$(d)$, $\bm=\bl$. Hence, by
Lemma~\ref{rankofL}$(a)$, $r=\rank(\bm)=\rank(\bl)=n-1$. Thus $J$,
which is generated by $r$ binomials defined by the $r$ columns of $M$,
can be generated by $r=n-1$ elements. If follows by
Proposition~\ref{gradeI}$(c)$ that $n$ must be equal to $2$. Hence, if
$n>2$, the intersection of $\mathfrak{A}$ and $\mathfrak{B}$ is empty.

Suppose now that $n=2$ and that $J$ is as above. By
Proposition~\ref{bm}$(b)$, $JB=I(M)B=I(\bm)B$, which, by hypothesis,
is prime. Therefore $S(J)=JB\cap A$ is also prime. By
Corollary~\ref{S(I)prime}, $\gcd(m(J))=1$. By Remark~\ref{casen=2},
$J=(x_1^{a_{1,1}}-x_2^{a_{2,2}})$, with $m(J)=(a_{2,2},a_{1,1})$. Thus
$\gcd(a_{1,1},a_{2,2})=1$. In particular,
$x_1^{a_{1,1}}-x_2^{a_{2,2}}$ is irreducible (see
\cite[Corollary~10.15]{fossum} or \cite[Lemma~8.2]{op2}; cf. also
Example~\ref{casen=2final}).

Conversely, suppose that $n=2$ and let $I=(f_1)$ be a principal ideal
generated by $f_1=x_1^{a}-x_2^{b}$, an irreducible binomial, i.e.,
$\gcd(a,b)=1$. In particular, $I$ is prime. Set
$M=(\begin{array}{cc}a&-b\end{array})^{\top}$ and complete $M$ to the
  obvious $2\times 2$ PCB matrix $L$. Clearly $M$ is a $2\times 1$
  matrix of rank $1$ with $\bm\cap \bn^2_0=\{ 0\}$ and
  $\bm=\bl$. Moreover, $I=I(M)=I(L)$. Thus $I$ is the PCB ideal
  associated to $L$. By Corollary~\ref{S(I)isI(bl)},
  $I(\bl)=S(I)$. But, by Remark~\ref{unmixedifn=3} and
  Proposition~\ref{S(I)}$(b)$, $I=S(I)$. Therefore,
  $I=S(I)=I(\bl)=I(\bm)$ and $I(\bm)$ is prime. Thus $I$ is in the
  intersection of $\mathfrak{A}$ and $\mathfrak{B}$.
\end{proof}

We now express $\fb_{\lambda}^{c}$ in terms of $\underline{x}$, the
original set of variables. We see that these prime ideals can be
expressed as the vanishing ideals of monomials curves ``with
coefficients''.

\begin{lemma}\label{alambda}
Let $PLQ=D$ be a normal decomposition of $L$ with
$p_{n,*}=(\nu_1,\ldots ,\nu_n)\in\bn^{n}$. Let $d_1,\ldots ,d_{n-1}$
be the invariant factors of $L$. Let $R=(r_{i,j})$ be the inverse of
$P$ and set $y_1=x^{r_{*,1}},\ldots ,y_n=x^{r_{*,n}}$ in
$B=k[\underline{x}^{\pm}]$. For any $\lambda=(\lambda_1,\ldots
,\lambda_{n-1})\in (k^{*})^{n-1}$, set
$\fa_{\lambda}=\ker(\varphi_{\lambda})$, where $\varphi_{\lambda}:A\to
k[t]$ is the natural map defined by the rule
$\varphi_{\lambda}(x_i)=\lambda_1^{p_{1,i}}\cdots
\lambda_{n-1}^{p_{n-1,i}}t^{\nu_i}$, for $i=1,\ldots ,n$. Then the
following hold.
\begin{itemize}
\item[$(a)$] $\fa_{\lambda}$ is a prime ideal of $A$ of height $n-1$;
\item[$(b)$] $\varphi_{\lambda}$ induces the morphism
  $\tilde{\varphi}_{\lambda}:B\to k[t,t^{-1}]$ that sends $y_{i}$ to
  $\lambda_i$, for $i=1,\ldots ,n-1$, and $y_n$ to $t$;
\item[$(c)$] $\fa_{\lambda}=\fb_{\lambda}\cap A=\fb_{\lambda}^{c}$,
  where $\fb_{\lambda}=(y_1-\lambda_1,\ldots
  ,y_{n-1}-\lambda_{n-1})B$.
\end{itemize}
\end{lemma}
\begin{proof}
Set $\theta_i=\lambda_1^{p_{1,i}}\cdots \lambda_{n-1}^{p_{n-1,i}}\in
k^{*}$.  Since $k[\theta_1t^{\nu_1},\ldots ,\theta_nt^{\nu_n}]\subset
k[t]$ is an integral extension, $A/\fa_{\lambda}\cong
k[\theta_1t^{\nu_1},\ldots ,\theta_nt^{\nu_n}]$ has Krull dimension
$1$, and $(a)$ follows. Notice that
\begin{eqnarray*}
&&\tilde{\varphi}_{\lambda}(y_{i})=\tilde{\varphi}_{\lambda}(x^{r_{*,i}})=
  (\tilde{\varphi}_{\lambda}(x_1))^{r_{1,i}}\cdots
  (\tilde{\varphi}_{\lambda}(x_n))^{r_{n,i}}=\\ &&\lambda_{1}^{p_{1,1}r_{1,i}+\ldots
    +p_{1,n}r_{n,i}}\cdots \lambda_{n-1}^{p_{n-1,1}r_{1,i}+\ldots
    +p_{n-1,n}r_{n,i}}t^{\nu_1r_{1,i}+\ldots +\nu_nr_{n,i}}.
\end{eqnarray*}
Since $PR$ is the identity matrix, the latter is equal to $\lambda_i$,
for $i=1,\ldots ,n-1$, and to $t$, for $i=n$. This proves
$(b)$. Finally, since $\fb_{\lambda}^{c}$ is a prime ideal of $A$ of
height $n-1$, to prove $(c)$ is enough to show
$\fb_{\lambda}^{c}\subseteq \fa_{\lambda}$.  Let $\sigma:A\to
B=S^{-1}A$ and $\rho:k[t]\to k[t,t^{-1}]$ be the canonical morphisms,
so $\tilde{\varphi}_{\lambda}\circ\sigma=
\rho\circ\varphi_{\lambda}$. Now, take $f\in\fb_{\lambda}^{c}$. Then
$f\in\fa_{\lambda}$ if and only if $\varphi_{\lambda}(f)=0$, and since
$\rho$ is injective, if and only if $\sigma(f)\in
\ker(\tilde{\varphi}_{\lambda})$. Since $\sigma(f)\in\fb_{\lambda}$,
it follows that $\sigma(f)=\sum_{i=1}^{n-1}g_i(y_i-\lambda_i)$, for
some $g_i\in B$. Thus $\tilde{\varphi}_{\lambda}(\sigma(f))=
\sum_{i=1}^{n-1} \tilde{\varphi}_{\lambda}(g_i)
\tilde{\varphi}_{\lambda}(y_i-\lambda_i)=0$ by $(b)$.
\end{proof}

We finish the section by stating the ``intrinsic'' role of the minimal
prime component $\fp_{m(I)}$, the unique Herzog ideal containing the
PCB binomial ideal $I$, among the other minimal primes of $I$. As was
to be expected, $\fp_{m(I)}$ is the minimal prime ideal picked out by
the element $(1,\ldots ,1)\in\Lambda(D)$, which exists for an
arbitrary coefficient field $k$.

\begin{remark}{\rm
Let $I$ be the PCB ideal associated to $L$, $m(I)$ its associated
integer vector, $d=\gcd(m(I))$ and $\nu(I)=m(I)/d$. Let $PLQ=D$ be a
normal decomposition of $L$ with $p_{n,*}=\nu(I)$. Let $d_1,\ldots
,d_{n-1}$ be the invariant factors of $L$.  Even if $k$ does not
contain the $d_{n-1}$-th roots of unity, we may write
$\{\xi_{i,1},\ldots ,\xi_{i,d_{i}}\}$ for the set of $d_i$-th roots of
unity in a field extension $\tilde{k}$ of $k$ (allowing possible
repetitions, by abuse of notation) and set
$\Lambda(D)=\prod_{i=1}^{n-1}\{\xi_{i,1},\ldots,\xi_{i,d_{i}}\}\subset
(\tilde{k}^{*})^{n-1}$ However, there is always one
$\lambda\in\Lambda(D)\cap (k^{*})^{n-1}$, namely $\lambda=(1,\ldots
,1)$. For this especial $\lambda$, $\varphi_{\lambda}:A\to k[t]$ sends
$x_i$ to $t^{\nu_i}$. Therefore, according to Lemma~\ref{alambda} and
Definition~\ref{defm(I)},
$\fa_{\lambda}=\ker(\varphi_{\lambda})=\fp_{\nu(I)}=\fp_{m(I)}$.
}\end{remark}

\section{Main Theorem}\label{mt}

We are now in position to state the main result of the paper,
recalling Theorem~\ref{embeddedcomp} for this purpose. As always,
$A=k[\underline{x}]$ is the polynomial ring in $n$ variables
$\underline{x}=x_1,\ldots ,x_n$ over a field $k$, $n\geq 2$,
$\fm=(\underline{x})$ is the maximal ideal generated by
$\underline{x}$, $S$ is the multiplicatively closed set generated by
$x=x_1\cdots x_n$ and $B=S^{-1}A=k[\underline{x}^{\pm}]=k[x_1,\ldots
  ,x_n,x_1^{-1},\ldots ,x_n^{-1}]$ is the Laurent polynomial ring. If
$I$ is an ideal of $A$, $IB$ denotes the extension of $I$ in $B$ and
$S(I)=IB\cap A$ the contraction of $IB$ in $A$.

\begin{theorem}\label{main}
Let $I$ be the PCB ideal associated to $L$, $m(I)$ its associated
integer vector, $d=\gcd(m(I))$ and $\nu(I)=m(I)/d$. Let $PLQ=D$ be a
normal decomposition of $L$ with $p_{n,*}=\nu(I)$. Let $d_1,\ldots
,d_{n-1}$ be the invariant factors of $L$.
\begin{itemize}
\item[$(a)$] Suppose that $k$ contains the $d_{n-1}$-th roots of unity
  and that the characteristic of $k$, $\ch(k)$, is zero or $\ch(k)=p$,
  $p$ a prime with $p\nmid d_{n-1}$. Write $\{\xi_{i,1},\ldots
  ,\xi_{i,d_{i}}\}$ to denote the set of $d_i$-th roots of unity in
  $k$ and
  $\Lambda(D)=\prod_{i=1}^{n-1}\{\xi_{i,1},\ldots,\xi_{i,d_{i}}\}$. For
  any $\lambda\in\Lambda(D)$, set
  $\fa_{\lambda}=\ker(\varphi_{\lambda})$, where
  $\varphi_{\lambda}:A\to k[t]$ is the natural map defined by the rule
  $\varphi_{\lambda}(x_i)=\lambda_1^{p_{1,i}}\cdots
  \lambda_{n-1}^{p_{n-1,i}}t^{\nu_i}$, for $i=1,\ldots ,n$. If $n\geq
  4$, $I=\cap_{\lambda\in\Lambda(D)}\fa_{\lambda}\cap \fc$, with $\fc$ an
  irredundant $\fm$-primary ideal, is a minimal primary decomposition
  of $I$ in $A$ and $I$ has exactly $d+1$ primary components. If
  $n\leq 3$, $I=\cap_{\lambda\in\Lambda(D)}\fa_{\lambda}$ is a minimal
  primary decomposition of $I$ in $A$, $I$ is radical and has exactly
  $d$ primary components.
\item[$(b)$] Suppose that $k$ is an arbitrary field. If $n\geq 4$, $I$
  is not unmixed and has at most $d+1$ primary components, only one of
  them embedded. If $n\leq 3$, $I$ is unmixed and has at most $d$
  primary components.
\end{itemize}
\end{theorem}
\begin{proof}
Let us show $(a)$. Let $R=(r_{i,j})$ be the inverse of $P$ and set
$y_1=x^{r_{*,1}},\ldots ,y_n=x^{r_{*,n}}$ in $B=k[\underline{x}^{\pm}]$.
By Theorem~\ref{dcompofSI}$(a)$, $IB=\cap
_{\lambda\in\Lambda(D)}\fb_{\lambda}$ is a minimal primary
decomposition of $IB$ in $B$, where
$\fb_{\lambda}=(y_1-\lambda_1,\ldots ,y_{n-1}-\lambda_{n-1})B$, and
$S(I)=\cap _{\lambda\in\Lambda(D)}\fb^{c}_{\lambda}$ is a minimal
primary decomposition of $S(I)$ in $A$. By Lemma~\ref{alambda},
$\fb_{\lambda}^{c}=\fa_{\lambda}$. Thus $S(I)=\cap
_{\lambda\in\Lambda(D)}\fa_{\lambda}$ is a minimal primary
decomposition of $S(I)$ in $A$. If $n\geq 4$, by
Proposition~\ref{mixedifn>3}, $I$ is not unmixed and, by
Proposition~\ref{S(I)}$(b)$, $I=S(I)\cap \fc$, with $\fc$ an $\fm$-primary
ideal, and this intersection is irredundant. Therefore
$I=\cap_{\lambda\in\Lambda(D)}\fa_{\lambda}\cap \fc$ is a minimal
primary decomposition of $I$ in $A$ and $I$ has exactly $d+1$ primary
components. On the other hand, if $n\leq 3$, by Remark~\ref{unmixedifn=3},
$I$ is unmixed and, by Proposition~\ref{S(I)}$(b)$,
$I=S(I)$. Therefore $I=\cap _{\lambda\in\Lambda(D)}\fa_{\lambda}$ is a
minimal primary decomposition of $S(I)$ in $A$, $I$ is radical and has
exactly $d$ primary components. Finally, $(b)$ follows from
Theorem~\ref{dcompofSI}$(b)$ and Proposition~\ref{S(I)}$(d)$.
\end{proof}

\begin{remark}\label{law}{\rm 
The thrust of \cite{lv} is to calculate the degree of a lattice ideal
that is homogeneous in the standard grading, whereas ours
(cf. Theorem~\ref{main} above) is to calculate the number of primary
components in a minimal binomial primary decomposition of a PCB ideal,
as well as to describe such components explicitly. The two enterprises
are of course linked to an extent by the Associativity Law of
Multiplicities.  }\end{remark}

\begin{example}\label{casen=2final}{\rm 
Let $I=(f_1,f_2)$ be the PCB ideal of $A$ associated to $L$,
$n=2$. Then $I=(x_1^{a_{1,1}}-x_2^{a_{2,2}})$,
$m(I)=(a_{2,2},a_{1,1})$ and $d=\gcd(m(I))=\gcd(a_{1,1},a_{2,2})$. Set
$a_{i,i}=da^{\prime}_{i,i}$ and $d=b_1a_{1,1}+b_2a_{2,2}$, for some
$b_1,b_2\in\bz$. Suppose that $k$ contains the $d$-th roots of unity
and that the characteristic of $k$, $\ch(k)$, is zero or $\ch(k)=p$,
$p$ a prime with $p\nmid d$.  Write $\Lambda(D)=\{\xi_{1},\ldots
,\xi_{d}\}$ to denote the set of $d$-th roots of unity in $k$. For any
$i=1,\ldots ,d$, set $\fa_i=\ker(\varphi_i)$, where $\varphi_i:A\to
k[t]$ is the natural map defined by the rule
$\varphi_i(x_1)=\xi_i^{b_{1}}t^{a^{\prime}_{2,2}}$ and
$\varphi_i(x_2)=\xi_i^{-b_{2}}t^{a^{\prime}_{1,1}}$. By
Example~\ref{Pn=2} and Theorem~\ref{main}, $I=\fa_1\cap \ldots \cap
\fa_d$ is a minimal primary decomposition of $I$ in $A$, $I$ is
radical and has exactly $d$ primary components.

Observe that each $\fa_i$ is a prime ideal of $A$ of height 1, hence
principal (see Lemma~\ref{alambda}). Clearly,
$x_1^{a^{\prime}_{1,1}}-\xi_ix_2^{a^{\prime}_{2,2}}$ is in $\fa_i$. A
variation of the argument in \cite[Corollary~10.15]{fossum} proves
that $x_1^{a^{\prime}_{1,1}}-\xi_ix_2^{a^{\prime}_{2,2}}$ is
irreducible. Alternatively, let $\tilde{k}$ be a field extension of
$k$ containing an $a_{2,2}^{\prime}$-th root $\eta$ of $\xi_i$. Set
$y_1=x_1$ and $y_2=\eta x_2$ and $A=k[x_1,x_2]\to
\tilde{k}[x_1,x_2]=\tilde{k}[y_1,y_2]=:C$, a flat extension. Set
$J=(x_1^{a^{\prime}_{1,1}}-\xi_ix_2^{a^{\prime}_{2,2}})C=
(y_1^{a^{\prime}_{1,1}}-y_2^{a^{\prime}_{2,2}})$, a PCB ideal in
$C=\tilde{k}[y_1,y_2]$, with
$\gcd(a_{1,1}^{\prime},a_{2,2}^{\prime})=1$. Applying
Theorem~\ref{main} or Corollary~\ref{S(I)prime} to $J$, one deduces
that $J$ is prime, hence $JC\cap
A=(x_1^{a^{\prime}_{1,1}}-\xi_ix_2^{a^{\prime}_{2,2}})$ is
prime. Therefore
$\fa_i=(x_1^{a^{\prime}_{1,1}}-\xi_ix_2^{a^{\prime}_{2,2}})$.  Since
$I =\fa_1\cap \ldots \cap \fa_d=\fa_1\cdots \fa_d$, the binomial
$x_1^{a_{1,1}}-x_2^{a_{2,2}}$ admits the decomposition
$\prod_{i=1}^{d} (x_1^{a^{\prime}_{1,1}}-\xi_ix_2^{a^{\prime}_{2,2}})$
as a product of irreducibles. In particular, any factor of
$x_1^{a_{1,1}}-x_2^{a_{2,2}}$ is of the form
$\prod_{j=1}^{r}(x_1^{a^{\prime}_{1,1}}-\xi_{i_j}x_2^{a^{\prime}_{2,2}})$,
where $1\leq r\leq d$, i.e., $x_1^{ra^{\prime}_{1,1}}+
\eta_1x_1^{(r-1)a^{\prime}_{1,1}}x_2^{a^{\prime}_{2,2}}+\ldots
+\eta_{r-1}x_1^{a^{\prime}_{1,1}}x_2^{(r-1)a^{\prime}_{2,2}}
+\eta_rx_2^{ra^{\prime}_{2,2}}$, for some $\eta_j\in \tilde{k}$, which
may or not be in $k$. This result recovers \cite[Lemma~8.2]{op2}.
}\end{example}

\begin{example}\label{simplestexn=4}
{\rm Let $I=(x_1^3-x_2x_3x_4,x_2^3-x_1x_3x_4,x_3^3-x_1x_2x_4,
  x_4^3-x_1x_2x_3)\subset A$, be the PCB ideal of
  Example~\ref{simplestn=4}. We know that $m(I)=(16,16,16,16)$ and
  $d=\gcd(m(I))=16$. Thus, by Theorem~\ref{main}, $I$ has at most
  seventeen primary components, one of them embedded. By
  Theorem~\ref{embeddedcomp}, $\fc=I+(x_2x_3^2)$ is an irredundant
  embedded $\fm$-primary component of $I$.

One can check that
\begin{eqnarray*}
\left(\begin{array}{rrrr}
1&0&0&0
\\-1&1&0&0
\\0&-1&1&0
\\1&1&1&1
\end{array}
\right)
\left(\begin{array}{rrrr}
3&-1&-1&-1
\\-1&3&-1&-1
\\-1&-1&3&-1
\\-1&-1&-1&3
\end{array}
\right)
\left(\begin{array}{rcccrrr}
1&2&1&1
\\1&3&1&1
\\1&3&2&1
\\0&0&0&1
\end{array}
\right)=
\left(\begin{array}{rcccrrr}
1&0&0&0
\\0&4&0&0
\\0&0&4&0
\\0&0&0&0
\end{array}
\right)
\end{eqnarray*}
is a normal decomposition of $L$. In particular, the invariant factors
of $L$ are $d_1=1, d_2=4, d_3=4$. Suppose that $k=\bc$.  Then
$\Lambda(D)=\{1\}\times \{1,i,-1,-i\}\times \{1,i,-1,-i\}\subset
\bc^{3}$. According to Theorem~\ref{main}, for a
$\lambda\in\Lambda(D)$, the natural morphism
$\varphi_{\lambda}:k[x_1,x_2,x_3,x_4]\to k[t]$ is defined by setting
$\varphi_{\lambda}(x_1)=\lambda_1\lambda_2^{-1}t$,
$\varphi_{\lambda}(x_2)=\lambda_2\lambda_3^{-1}t$,
$\varphi_{\lambda}(x_3)=\lambda_3t$ and $\varphi_{\lambda}(x_4)=t$. To
simplify notations we just write the ordered $4$-tuple
$(\lambda_1\lambda_2^{-1}t, \lambda_2\lambda_3^{-1}t,\lambda_3t,t)$ to
decribe the morphism $\varphi_{\lambda}$. Using this notation, the
sixteen morphisms are the following:
\begin{eqnarray*}
&&(t,t,t,t),(t,-it,it,t),(t,-t,-t,t),(t,it,-it,t),
  (-it,it,t,t),(-it,t,it,t),\\&&(-it,-it,-t,t),(-it,-t,-it,t),
  (-t,-t,t,t),(-t,it,it,t),(-t,t,-t,t),\\&&(-t,-it,-it,t),
  (it,-it,t,t),(it,-t,it,t),(it,it,-t,t),(it,t,-it,t).
\end{eqnarray*}
Therefore, $I=\cap_{i=1}^{16}\fa_i\cap \fc$, where the sixteen minimal
primary components $\fa_i$ are the kernels of the preceding morphisms:
\begin{eqnarray*}
&&\fa_1=(x_1-x_4,x_2-x_4,x_3-x_4), 
\fa_2=(x_1-x_4,x_2+ix_4,x_3-ix_4),
  \\&&\fa_3=(x_1-x_4,x_2+x_4,x_3+x_4),
  \fa_4=(x_1-x_4,x_2-ix_4,x_3+ix_4),
  \\&&\fa_5=(x_1+ix_4,x_2-ix_4,x_3-x_4),
  \fa_6=(x_1+ix_4,x_2-x_4,x_3-ix_4),
  \\&&\fa_7=(x_1+ix_4,x_2+ix_4,x_3+x_4),
  \fa_8=(x_1+ix_4,x_2+x_4,x_3+ix_4),
  \\&&\fa_9=(x_1+x_4,x_2+x_4,x_3-x_4),
  \fa_{10}=(x_1+x_4,x_2-ix_4,x_3-ix_4),
  \\&&\fa_{11}=(x_1+x_4,x_2-x_4,x_3+x_4),
  \fa_{12}=(x_1+x_4,x_2+ix_4,x_3+ix_4),
  \\&&\fa_{13}=(x_1-ix_4,x_2+ix_4,x_3-x_4),
  \fa_{14}=(x_1-ix_4,x_2+x_4,x_3-ix_4),
  \\&&\fa_{15}=(x_1-ix_4,x_2-ix_4,x_3+x_4),
  \fa_{16}=(x_1-ix_4,x_2-x_4,x_3+ix_4).
\end{eqnarray*}
Let us obtain the minimal primary components of $I$ over $\br$ (and
similarly over $\bq$). Consider the ideal
$I_{2,4}:=(x_1-x_4,x_2+x_3,x_3^2+x_4^2)$ in
$A=\bc[\underline{x}]$. Clearly, $A/I_{2,4}\cong
\bc[x_3,x_4]/(x_3^2+x_4^2)$, so $I_{2,4}$ is a complete intersection
of height 3, in particular, unmixed. Moreover, $I_{2,4}\subseteq
\fa_2\cap \fa_4$, and if $\fp$ is a prime over $I_{2,4}$ we see that
$\fp$ contains $\fa_2$ or $\fa_4$. Thus $\fa_2$ and $\fa_4$ are the
associated primes of $I_{2,4}$. Since $x_3+ix_4\not\in\fa_2$,
$x_3-ix_4$ and $x_3^2+x_4^2$ are associated in $A_{\mathfrak{a}_2}$
and
$(I_{2,4})_{\mathfrak{a}_2}=(\fa_2\cap\fa_4)_{\mathfrak{a}_2}$. Analogously
$(I_{2,4})_{\mathfrak{a}_4}=(\fa_2\cap\fa_4)_{\mathfrak{a}_4}$. Therefore
$I_{2,4}=\fa_2\cap \fa_4$. Similarly, we have
\begin{eqnarray*}
&&I_{5,13}:=(x_1+x_2,x_2^2+x_4^2,x_3-x_4)=\fa_5\cap\fa_{13},\\&&
  I_{6,16}:=(x_1+x_3,x_2-x_4,x_3^2+x_4^2)=\fa_6\cap\fa_{16},\\&&
  I_{7,15}:=(x_1-x_2,x_2^2+x_4^2,x_3+x_4)=\fa_7\cap\fa_{15},\\&&
  I_{8,14}:=(x_1-x_3,x_2+x_4,x_3^2+x_4^2)=\fa_8\cap\fa_{14}\mbox{ and
  }\\&&I_{10,12}:=(x_1+x_4,x_2-x_3,x_3^2+x_4^2)=\fa_{10}\cap\fa_{12}.
\end{eqnarray*}
Therefore, $I=\fa_1\cap\fa_3\cap\fa_9\cap\fa_{11}\cap I_{2,4}\cap
I_{5,13}\cap I_{6,16}\cap I_{7,15}\cap I_{8,14}\cap I_{10,12}\cap
\fc$. Note that the ideals appearing in this expression are generated
by binomials with coefficients in $\br$. Let us momentarily denote by
$\tilde{I}$, $\tilde{\fa}_i$, $\tilde{I}_{i,j}$ and $\tilde{\fc}$ the
corresponding ideals considered in $R=\br[\underline{x}]$, i.e.,
$\tilde{I}=(x_1^3-x_2x_3x_4,x_2^3-x_1x_3x_4,x_3^3-x_1x_2x_4,
x_4^3-x_1x_2x_3)R$, $\tilde{\fa}_1=(x_1-x_4,x_2-x_4,x_3-x_4)R$ and so
on.  Clearly, their extension in $A$ are the original ideals, i.e.,
$\tilde{I}A=I$, $\tilde{\fa}_iA=\fa_i$, $\tilde{I}_{i,j}A=I_{i,j}$ and
$\tilde{\fc}A=\fc$. Moreover, since $R\to A$ is faithfully flat,
$\tilde{I}=\tilde{I}A\cap R=I\cap R$,
$\tilde{\fa}_i=\tilde{\fa}_iA\cap R=\fa_i\cap R$,
$\tilde{I}_{i,j}=\tilde{I}_{i,j}A\cap R=I_{i,j}\cap R$ and
$\tilde{\fc}=\tilde{\fc}A\cap R=\fc\cap R$.

Hence $\tilde{I}=\tilde{\fa}_1\cap\tilde{\fa}_3\cap\tilde{\fa}_9\cap
\tilde{\fa}_{11}\cap\tilde{I}_{2,4}\cap\tilde{I}_{5,13}\cap\tilde{I}_{6,16}
\cap\tilde{I}_{7,15}\cap\tilde{I}_{8,14}\cap\tilde{I}_{10,12}\cap\tilde{\fc}$
is a minimal primary decompostion of $\tilde{I}$ in
$R=\br[\underline{x}]$. Indeed, $\tilde{\fa}_{i}$ is a prime ideal of
$R$ for $i=1,3,9$ and $11$. Moreover,
$R/\tilde{I}_{2,4}\cong\br[x_3,x_4]/(x_3^2+x_4^2)$, a domain, so
$\tilde{I}_{2,4}=(x_1-x_4,x_2+x_3,x_3^2+x_4^2)$ is a prime ideal of
$R$. Analogously, $I_{5,13}$, $I_{6,16}$ $I_{7,15}$, $I_{8,14}$ and
$I_{10,12}$ are prime ideals of $R$. Moreover, applying
Theorem~\ref{embeddedcomp} to the PCB ideal $\tilde{I}$ of
$R=\br[\underline{x}]$, one obtains $\tilde{\fc}$ as an irredundant
embedded primary component of $\tilde{I}$. Finally, the full
decomposition is irredundant because all the primes $\tilde{\fa}_i$
and $\tilde{I}_{i,j}$ appearing are different and of the same height.

Now suppose that $k=\bz/2\bz$. As before, $I$ is not unmixed,
$\fc=I+(x_2x_3^2)$ is an irredundant embedded component of $I$ and
$d_1=1$, $d_2=4$ and $d_3=4$ are the invariant factors of $L$. By
Lemma~\ref{easyexpofIB},
$IB=(y_1-1,y_2^4-1,y_3^4-1)B=(y_1-1,(y_2-1)^4,(y_3-1)^4)B$.  For
$\lambda=(1,1,1)$, we clearly have $\fb_{\lambda}^4\subsetneq
IB\subsetneq \fb_{\lambda}$, where $\fb_{\lambda}=(y_1-1,y_2-1,y_3-1)$
(see Notation~\ref{blambda}). By Lemma~\ref{alambda},
$\fb_{\lambda}^c=\fa_{\lambda}$ and
$\fa_{\lambda}=(x_1-x_4,x_2-x_4,x_3-x_4)$. Hence
$\fa_{\lambda}^4\subseteq (\fb_{\lambda}^4)^c\subseteq S(I)\subseteq
\fa_{\lambda}$.  Since $S(I)$ is unmixed (cf. Proposition~\ref{S(I)}),
$S(I)$ is an $\fa_{\lambda}$-primary ideal. Therefore $I$ has exactly
two primary components, namely $\fa_{\lambda}$ and $\fm$.
}\end{example}

\begin{example}{\rm 
As a generalization of Example~\ref{simplestexn=4}, for $n\geq 3$, let
\begin{eqnarray*}
I=(x_1^{n-1}-x_2\cdots x_n,\ldots ,x_n^{n-1}-x_1\cdots x_{n-1})
\end{eqnarray*}
be the PCB ideal associated to the $n\times n$ PCB matrix $L$ with
diagonal entries $n-1$ and off-diagonal entries $-1$. One can check
that the invariant factors of $L$ are $d_1=1,d_2=n,\ldots ,d_{n-1}=n$
and that a normal decomposition of $PLQ=D$ is given by
\begin{eqnarray*}
P=\left(\begin{array}{rrrcrrr}
1&0&0&\ldots&0&0&0
\\-1&1&0&\ldots&0&0&0
\\0&-1&1&\ldots&0&0&0
\\\vdots&\vdots&\vdots&&\vdots&\vdots&\vdots
\\0&0&0&\ldots&1&0&0
\\0&0&0&\ldots&-1&1&0
\\1&1&1&\ldots&1&1&1
\end{array}
\right)\mbox{ and }
Q=\left(\begin{array}{rcccrrr}
1&n-2&n-3&\ldots&2&1&1
\\1&n-1&n-3&\ldots&2&1&1
\\1&n-1&n-2&\ldots&2&1&1
\\\vdots&\vdots&\vdots&&\vdots&\vdots&\vdots
\\1&n-1&n-2&\ldots&3&1&1
\\1&n-1&n-2&\ldots&3&2&1
\\0&0&0&\ldots&0&0&1
\end{array}
\right).
\end{eqnarray*}
In particular, $d=d_1\ldots d_{n-1}=n^{n-2}$ and $\nu(I)=(1,\ldots
,1)$. Therefore, by Theorem~\ref{main}, $I$ has at most
$d+1=n^{n-2}+1$ prime components. Suppose that $k=\bc$. For
$\lambda\in\Lambda(D)$, let $\fa_{\lambda}=\ker(\varphi_{\lambda })$
where $\varphi_{\lambda}:A\to k[t]$ is the natural map defined by the
rule $\varphi_{\lambda}(x_{i})=\lambda_i\lambda_{i+1}^{-1}t$, for
$i=1,\ldots ,n-2$, $\varphi_{\lambda}(x_{n-1})=\lambda_{n-1}t$ and
$\varphi_{\lambda}(x_n)=t$. Then each $\fa_{\lambda}$ is a prime
ideal. If $n=3$, $I=\cap_{\lambda\in\Lambda(D)}\fa_{\lambda}$, whereas
if $n\geq 4$, $I=\cap_{\lambda\in\Lambda(D)}\fa_{\lambda}\cap \fc$ with
$\fc=I+(x^{b(n)})$ an $\fm$-primary ideal; in each case, these
decompositions are irredundant.}\end{example}

\noindent {\sc Acknowledgement}. We gratefully acknowledge Josep
\`Alvarez Montaner's interest and our useful discussions with him. We
are grateful to Rafael H. Villarreal for his stimulating and generous
comments. The authors gratefully acknowledge financial support from
the research grant MTM2010-20279-C02-01 and the Universitat
Polit\`ecnica de Catalunya during the development of this
research. The first author benefitted from the hospitality offered by
the Departament de Matem\`atica Aplicada 1 at UPC during his research
visit there.

\renewcommand{\baselinestretch}{0.8}
\begin{center}
\parbox[t]{7cm}{\footnotesize

\noindent Liam O'Carroll

\noindent Maxwell Institute for Mathematical Sciences

\noindent School of Mathematics

\noindent University of Edinburgh

\noindent EH9 3JZ, Edinburgh, Scotland

\noindent L.O'Carroll@ed.ac.uk } 
\,
\parbox[t]{6.2cm}{\footnotesize

\noindent Francesc Planas-Vilanova

\noindent Departament de Matem\`atica Aplicada~1

\noindent Universitat Polit\`ecnica de Catalunya

\noindent Diagonal 647, ETSEIB

\noindent 08028 Barcelona, Catalunya

\noindent francesc.planas@upc.edu}
\end{center}

{\small
\begin{thebibliography}{cccccc}
\bibitem[AV]{av}{A. Alc\'antar, R.H. Villarreal, Critical binomials of
  monomials curves, Comm. Algebra {\bf 22}(8) (1994), 3037-3052.}
\bibitem[BH]{bh}{W. Bruns, J. Herzog, Cohen-Macaulay rings, Cambridge
  studies in advanced mathematics 39, Cambridge University Press
  1993.}
\bibitem[CLO]{clo}{D. Cox, J. Little, D. O'Shea, Using algebraic
  geometry, Second edition. Graduate Texts in Mathematics, {\bf
    185}. Springer, New York, 2005.}
\bibitem[DMM]{dmm}{A. Dickenstein, L. Matusevich, E. Miller,
  Combinatorics of binomial primary decomposition. Math. Z. {\bf 264}
  (2010), no. 4, 745-763.}
\bibitem[Eis]{eisenbud}{D. Eisenbud, Commutative algebra with a view toward
  Algebraic Geometry, Graduate Texts in Mathematics, {\bf
    150}. Springer-Verlag, New York, 1995.}
\bibitem[ES]{es}{D. Eisenbud, B. Sturmfels, Binomial ideals. Duke
  Math. J. {\bf 84} (1996), no. 1, 1-45.}
\bibitem[Eto]{eto}{K. Eto, Almost complete intersection monomial
  curves in $\ba^{4}$. Comm. Algebra {\bf 22} (1994), no. 13,
  5325-5342.}
\bibitem[Fos]{fossum}{R.M. Fossum, The divisor class group of a Krull
  domain, Ergebnisse der Mathematik und ihrer Grenzgebiete, Volume 74
  (Springer, 1973).}
\bibitem[Gas]{gastinger}{W. Gastinger, \"Uber die Verschwindungsideale
  monomialer Kuerven. PhD dissertation. Universit\"at Regensburg,
  1989.}
\bibitem[GPS]{singular}{G.-M.~Greuel, G.~Pfister, and H.~Sch\"onemann.
  {{\sc Singular}}. A Computer algebra system for polynomial
  computations. University of
  Kaiserslautern. {http://www.singular.uni-kl.de}.}
\bibitem[HMV]{hmv}{M. Herrmann, B. Moonen, O. Villamayor, Ideals of linear
type and some variants.  The Curves Seminar at Queen's, Vol. VI
(Kingston, ON, 1989), Exp. No. H, 37 pp., Queen's Papers in Pure and
Appl. Math., 83, Queen's Univ., Kingston, ON, 1989.}
\bibitem[Her]{herzog}{J. Herzog, Generators and relations of abelian
  semigroups and semigroup rings. Manuscripta Math. {\bf 3} (1970),
  175-193.}
\bibitem[HS]{hs}{S. Ho\c{s}ten, J. Shapiro, Primary decomposition of
  lattice basis ideals, J. Symbolic Computation (2000) {\bf 29},
  625-639.}
\bibitem[Jac]{jacobson}{N. Jacobson, Basic algebra. I. Second
  edition. W. H. Freeman and Company, New York, 1985.}
\bibitem[KM]{km}{T. Kahle, E. Miller, Decompositions of commutative
  monoid congruences and binomial ideals, arXiv:1107.4699.}
\bibitem[Kap]{kaplansky}{I. Kaplansky, Commutative rings. Revised
  edition. The University of Chicago Press, Chicago, Ill.-London,
  1974.} 
\bibitem[KO]{ko}{A. Katsabekis, I. Ojeda, An indispensable
  classification of monomial curves in $\ba^{4}_{k}$, arXiv:1103.4702.}
\bibitem[LV]{lv}{H.H. L\'opez, R.H. Villarreal, Computing the degree
  of a lattice ideal of dimension one, arXiv:1206.18992.}
\bibitem[Mil]{miller}{E. Miller, Theory and applications of lattice
  point methods for binomial ideals, arXiv:1009.2823.}
\bibitem[MS]{ms}{E. Miller, B. Sturmfels, Combinatorial commutative
  algebra, Graduate Texts in Mathematics, {\bf 227}. Springer-Verlag,
  New York, 2005}
\bibitem[Nor]{northcott}{D.G. Northcott, A homological investigation of a
certain residual ideal, Math. Ann. {\bf 150} (1963), 99-110.}
\bibitem[OP1]{op1}{L. O'Carroll, F. Planas-Vilanova, Irreducible
  affine space curves and the uniform Artin-Rees property on the prime
  spectrum, J. Algebra {\bf 320} (2008), 3339-3344.}
\bibitem[OP2]{op2}{L. O'Carroll, F. Planas-Vilanova, Ideals of
  Herzog-Northcott type. Proc. Edinb. Math. Soc. (2) {\bf 54} (2011),
  no.~1, 161-186.}
\bibitem[Oje]{ojeda}{I. Ojeda, Binomial canonical decompositions of binomial
ideals, Comm. Algebra {\bf 39} (2011), no. 10, 3722-3735.}
\bibitem[RVZ]{rvz}{E. Reyes, R.H. Villarreal, L. Z\'arate, A note on
  affine toric varieties, Linear Algebra Appl. {\bf 318} (2000),
  no. 1-3, 173-179.}
\bibitem[Vas]{vasconcelos}{W.V. Vasconcelos, Computational methods in
  commutative algebra and algebraic geometry. Algorithms and
  Computation in Mathematics, 2. Springer-Verlag, Berlin, 1998.}
\bibitem[Vil]{villarreal}{R.H. Villarreal, Monomial
  algebras. Monographs and Textbooks in Pure and Applied Mathematics,
  {\bf 238}. Marcel Dekker, Inc., New York, 2001.}
\bibitem[Wal]{waldi}{R. Waldi, Zur Konstruktion von Weierstrasspunkten
  mit Vorgegebener Halbgruppe, Manuscripta Math. {\bf 30} (1980),
  257-278.}
\bibitem[ZS]{zs}{O. Zariski, P. Samuel, Commutative Algebra, Vol.~II, 
Reprint of the 1960 edition. Graduate Texts in Mathematics {\bf 29},
Springer-Verlarg, New York-Heidelberg, 1975.}
\end{thebibliography}}
\end{document}